\documentclass[aps,reprint,superscriptaddress,letterpaper,onecolumn,floatfix]{revtex4-2}

\usepackage{graphicx}
\usepackage{esvect}
\usepackage{color}
\usepackage{bm}
\usepackage{amsmath,amsthm}
\newcommand{\erf}[1]{\ensuremath{\mathrm{erf}\left[#1\right]}}

\newcommand{\vect}[1]{\ensuremath{{\bm{#1}}}}

\newcommand{\smallo}

\theoremstyle{plain}
\newtheorem{theorem}{Theorem}

\newtheorem*{remark}{Remark}
\newtheorem{lemma}[theorem]{Lemma}

\newtheorem{assumption}{Assumption}

\renewcommand{\P}{\operatorname{\mathbb{P}}}
\newcommand{\E}{\operatorname{\mathbb{E}}}

\usepackage[mathscr]{euscript}
\DeclareSymbolFont{rsfs}{U}{rsfs}{m}{n}
\DeclareSymbolFontAlphabet{\mathscrsfs}{rsfs}

\newcommand{\one}{\mathbf{1}}
\newcommand{\R}{\mathbb{R}}
\newcommand{\Z}{\mathbb{Z}}
\newcommand{\bg}{\bm{g}}
\newcommand{\bx}{\bm{x}}
\newcommand{\by}{\bm{y}}
\newcommand{\bw}{\bm w}
\newcommand{\bG}{\bm{G}}

\newcommand{\bI}{\bm{I}}

\def\sSAT{\mbox{\rm\tiny SAT}}

\newcommand{\rmd}{\mathrm{d}}
\renewcommand{\hat}{\widehat}

\renewcommand{\pl}{\mbox{\rm\tiny pl}}
\newcommand{\rd}{\mbox{\rm\tiny rd}}

\newcommand{\kold}{\kappa_0}
\newcommand{\knew}{\kappa}
\newcommand{\tkold}{\tilde{\kappa}_0}
\newcommand{\tknew}{\tilde{\kappa}}
\newcommand{\kenerg}{\kappa_{\rm ener}}
\newcommand{\kentro}{\kappa_{\rm entr}}
\newcommand{\tkenerg}{\tilde{\kappa}_{\rm ener}}
\newcommand{\tkentro}{\tilde{\kappa}_{\rm entr}}

\usepackage[bitstream-charter]{mathdesign}
\DeclareSymbolFont{usualmathcal}{OMS}{cmsy}{m}{n}
\DeclareSymbolFontAlphabet{\mathcal}{usualmathcal}

\graphicspath{{./Figures/}}

\begin{document}

\begin{center}{\Large \textbf{
On the Atypical Solutions of the Symmetric Binary Perceptron}}\end{center}

\begin{center}
Damien Barbier\textsuperscript{1},
Ahmed El Alaoui\textsuperscript{2},
Florent Krzakala\textsuperscript{1} and
Lenka Zdeborov\'a\textsuperscript{3}
\end{center}

\begin{center}
{\bf 1} Information Learning \& Physics laboratory, 
\'Ecole Polytechnique F\'ed\'erale de Lausanne, Switzerland
\\
{\bf 2} Department of Statistics and Data Science, Cornell University, Ithaca, NY, USA
\\
{\bf 3} Statistical Physics of Computation laboratory, 
\'Ecole Polytechnique F\'ed\'erale de Lausanne, Switzerland 
\\

\end{center}

\begin{abstract}
We study the random binary symmetric perceptron problem, focusing on the behavior of rare high-margin solutions. 
While most solutions are isolated, we demonstrate that these rare solutions are part of clusters of extensive entropy, heuristically corresponding to non-trivial fixed points of an approximate message-passing algorithm. We enumerate these clusters via a local entropy, defined as a Franz-Parisi potential, which we rigorously evaluate using the first and second moment methods in the limit of a small constraint density $\alpha$ (corresponding to vanishing margin $\kappa$) under a certain assumption on the concentration of the entropy. This examination unveils several intriguing phenomena: i) We demonstrate that these clusters have an entropic barrier in the sense that the entropy as a function of the distance from the reference high-margin solution is non-monotone when $\knew \le 1.429 \sqrt{-\alpha/\log{\alpha}}$, while it is monotone otherwise, and that they have an energetic barrier in the sense that there are no solutions at an intermediate distance from the reference solution when $\knew \le 1.239 \sqrt{-\alpha/ \log{\alpha}}$. 
The critical scaling of the margin $\kappa$ in $\sqrt{-\alpha/\log\alpha}$ corresponds to the one obtained from the earlier work of Gamarnik et al.~\cite{gamarnik2022algorithms} for the overlap-gap property, a phenomenon known to present a barrier to certain efficient algorithms. 
ii) We establish using the replica method that the complexity (the logarithm of the number of clusters of such solutions) versus entropy (the logarithm of the number of solutions in the clusters) curves are partly non-concave and correspond to very large values of the Parisi parameter, with the equilibrium being reached when the Parisi parameter diverges.
\end{abstract}

\maketitle

\section{Introduction}
\subsection{Background and Motivation}
We consider the symmetric binary perceptron (SBP), introduced in \cite{aubin2019storage}, where we let $\bG = (\bg_a)_{a=1}^M$ be a collection of $M$ i.i.d.\ standard Gaussian random vectors in $\R^N$, with $M=\lfloor\alpha N \rfloor$ for a fixed $\alpha >0$ and for $\knew >0$, we consider the set of binary solutions $\bx \in \{-1,+1\}^N$ to the system of linear inequalities
\begin{equation}\label{eq:perceptron}
\big| \langle \bg_a , \bx \rangle \big| \le \knew \sqrt{N} ~~~~~ \mbox{for all}~ 1 \le a \le M\, .
\end{equation}    
We denote the set of solutions by $S(\bG, \knew)$, and its cardinality by  
\begin{equation}
Z(\bG,\knew) = \big|S(\bG, \knew)\big|\, .
\end{equation} 
It was shown by Aubin, Perkins and Zdeborov\'a~\cite{aubin2019storage} that $S(\bG,\knew)$ is nonempty with high probability if and only if $\kappa > \kappa_{\sSAT}(\alpha)$ where $\kappa_{\sSAT}(\alpha)$ is defined by the equation 
\begin{equation}
\P\big(|Z| \le \kappa\big)^{\alpha} = 1/2\, ,~~~~ Z \sim N(0,1)\, .
\end{equation}
Moreover, in the limit of small $\alpha$ we have 
\begin{equation}
\kappa_{\sSAT}(\alpha) \underset{\alpha\rightarrow0}{\sim}\sqrt{\frac{\pi}{2}}\, 2^{-1/\alpha}\, .
\end{equation}
Our main interest is in investigating the possibility of finding solutions efficiently when $\kappa > \kappa_{\sSAT}(\alpha)$.  

M\'ezard and Krauth \cite{krauth89storage} showed in their seminal work using the non-rigorous replica method \cite{mezard1987spin} that the solution landscape of the one-sided perceptron (where there is no absolute value in the constraints~\eqref{eq:perceptron}) is dominated by \emph{isolated} solutions lying at large mutual Hamming distances, a structure sometimes called ``frozen replica symmetry breaking'' \cite{martin2004frozen,zdeborova2008locked,huang2013entropy,huang2014origin,Semerjian}.  From the mathematics point of view, 
the frozen replica symmetry breaking prediction was proven true for the SBP in works by Perkins and Xu~\cite{perkins2021frozen} and Abb\'e, Li and Sly~\cite{abbe2022proof}, who showed that for all $\kappa > \kappa_{\sSAT}(\alpha)$, a solution drawn uniformly at random from $S(\bG,\knew)$ is \emph{isolated} with high probability, in the sense that it is separated from any other solution by a Hamming distance linear in $N$.

This type of landscape property has been traditionally associated with algorithmic hardness, with the rationale that an algorithm performing local moves is unlikely to succeed in the face of such extreme clustering, as argued, for instance, by Zdeborov\'a and M\'ezard~\cite{zdeborova2008constraint}, or Huang and Kabashima~\cite{huang2014origin}. In some problems, this predicted algorithmic hardness was confirmed empirically, e.g. \cite{zdeborova2008constraint,zdeborova2008locked}. In other problems, a prominent example being the binary perceptron (symmetric or not), it is known that certain efficient heuristics are able to find solutions for $\alpha$ small enough as a function of $\knew$~\cite{kim1998covering,braunstein2006learning,baldassi2007efficient,huang2013entropy,baldassi2015max,erba2023statistical,algo_fede}. 
Statistical physics studies of the neighborhood of the solutions returned by efficient heuristics have put forward the intriguing observation that  
in the binary perceptron problem, a dense region of other solutions surrounds the ones which are returned~\cite{baldassi2015subdominant,baldassi2016unreasonable,baldassi2016local}. 
This means that efficient algorithms may be drawn to \emph{rare, well connected} subset(s) of $S(\bG,\knew)$. 
Moreover, these efficient algorithms fail to return a solution when $\alpha$ becomes large, suggesting the existence of a \emph{computational phase transition} in the binary perceptron (symmetric or not). 

For the symmetric version of the problem, this state of affairs has been partially elucidated in two recent mathematical works: In~\cite{abbe2022binary}, Abb\'e et al.\  show the existence of clusters of solutions of linear diameter for all $\kappa > \kappa_{\sSAT}(\alpha)$, and maximal diameter for $\alpha$ small enough. In a different direction, Gamarnik et al.~\cite{gamarnik2022algorithms} established an almost sharp result in the regime of small $\alpha$, stating the following: There exists constants $c_0, c_1 >0 $ such that for $\alpha$ small enough, 
\begin{itemize}
\item if $\kappa \ge  c_0 \sqrt{\alpha}$ 
then a certain online algorithm of Bansal and Spencer~\cite{bansal2020line} finds a solution in $S(\bG,\kappa)$, and         
\item if $\kappa \le  c_1 \sqrt{-\alpha/\log(\alpha)}$ %
then $S(\bG,\kappa)$ exhibits a \emph{overlap gap property} ruling out a wide class  of efficient algorithms. 
\end{itemize}
We mention that the positive result which holds for $\kappa \ge  c_0 \sqrt{\alpha}$ is established in the case where the constraint matrix $\bG$ is Rademacher instead of Gaussian; nevertheless, the same result is expected in the Gaussian case.

Baldassi et al.~\cite{baldassi2020clustering} suggest that this computational transition can be probed by studying the monotonicity properties of the \emph{local entropy} of solutions around atypical solutions $\bx_0$ as a function of the distance from this solution. One can interpret the results of~\cite{baldassi2020clustering} as evidence towards a conjecture that finding a solution is computationally easy precisely when there exist some rare solutions around which this local entropy is monotone in the distance and that the problem becomes hard when this local entropy develops a local maximum at some distance $r_0$ from the reference solution $\bx_0$.  
If such a conjecture is correct, then it must agree with the above-mentioned finding of Gamarnik et al.~\cite{gamarnik2022algorithms} in the regime of small $\alpha$. This question motivated the present work.  

Another gap in the physics literature we elucidate in this work relates to the fact that the replica method on the one-step replica symmetry breaking level so far has not managed to find clusters of solutions in the binary perceptron. Indeed, the method can count rare clusters as long as they correspond to fixed points of a corresponding message-passing algorithm, see e.g.~\cite{zdeborova2007phase}. Parallels between the 1RSB calculation and the analysis of solutions with a monotonic local entropy have been put forward in \cite{baldassi2015subdominant,baldassi2016unreasonable,baldassi2016local}, but not in the form where one writes the standard 1RSB equations and shows that they have a solution corresponding to rare subdominant clusters. We show that the standard 1RSB framework actually does present such solutions which describe subdominant clusters of extensive entropy, and we give likely reasons why these solutions were missed in past investigations.

\subsection{Summary of our results}

\paragraph{Local entropy around high margin solutions:} We define and study a notion of local entropy around solutions which are typical at some margin $\kold <  \knew$. While typical solutions at $\kold$ are isolated from each other, it was shown in \cite{abbe2022binary} that they belong to connected components of solutions at margin $\knew$ having a linear diameter in $N$. Here, we show that these solutions are surrounded by exponentially many solutions at margin $\knew$.  

Consistently with the statistical physics literature, we say that there is a cluster of extensive entropy around a reference solution $\bx_0$ when the local entropy as a function of the distance achieves a local maximum at some distance from $\bx_0$. 
We show that for a certain range of $\knew$ typical solutions at margin $\kappa_0$ have extensive entropy clusters around them. We define the entropy of these clusters as the value of the entropy at a local maximum. An analogous investigation of local entropy around large margin solutions was performed in \cite{baldassi2021unveiling} for the one-sided binary perceptron using the replica method.

In our case, the symmetry of the constraints~\eqref{eq:perceptron} allows us to derive simpler formulas for the local entropy in the regime of small $\alpha$, essentially via a first moment method. This is due to the present model being contiguous to a corresponding simpler planted model in which the first and second moment computations can be conducted. We show that under a certain assumption on the concentration of the entropy of the SBP, while for any constant value of $\alpha$ the second moment is exponentially larger than the square of the first moment, the exponent of the ratio of these quantities, when normalized by $N$, tends to zero in the limit of small $\alpha$.

 The resulting entropy of these clusters is plotted in Fig.~\ref{fig:different_kappa0} for various values of $\knew$ and $\kold$ in the $\alpha\to 0$ limit. We observe that at a certain margin $\knew_{\rm entr}(\kold)$ the entropy curve stops because the local entropy curve becomes monotone in the distance for $\knew > \knew_{\rm entr}(\kold)$.
 As discussed above, the existence of reference solutions such that the local entropy curve is monotone was speculated to provoke the onset of a region of parameters where finding solutions is algorithmically easy. In this paper, we show the existence of solutions--those typical at $\kold$--for which the local entropy is monotone, and hence we do not expect the problem to be computationally hard for $\knew > \knew_{\rm entr}(\kold)$.
 In Fig.~\ref{fig:different_kappa0} we see that the smallest $\knew$ where this happens is $\knew_{\rm entr}  \equiv {\rm min}_{\kold} \knew_{\rm entr}(\kold) =\knew_{\rm entr}(\kold =\kappa_{\sSAT})$. 
 For this reason, a large part of this investigation is devoted to the case $\kold =\kappa_{\sSAT}(\alpha)$.

\begin{figure}[t]
    \centering
    \includegraphics[width=0.5\textwidth]{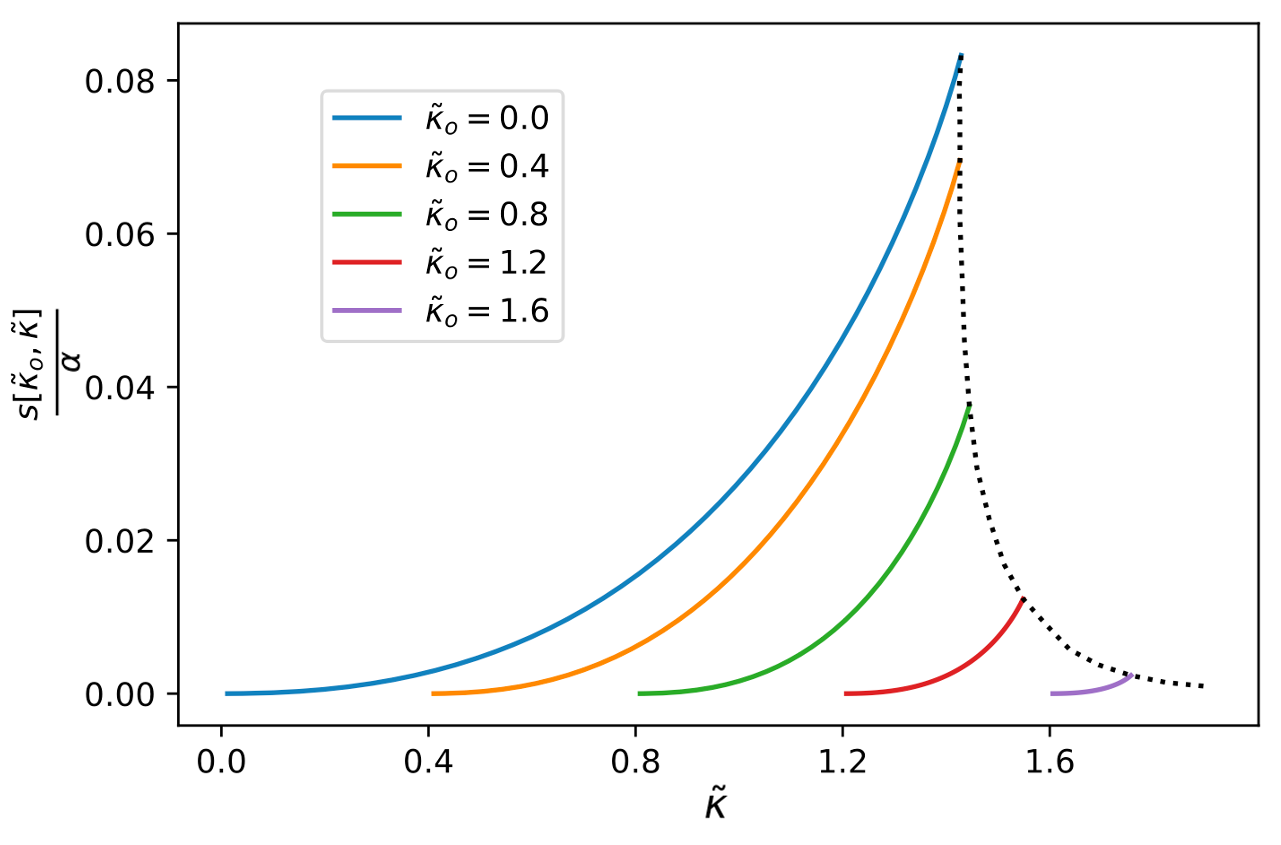}
     \caption{Entropy of clusters that exist at margin $\knew$ around a typical solution at margin $\kold$. We focus on the small $\alpha$ limit where the margins are rescaled as 
     $\knew=\tknew\sqrt{-\alpha/\log(\alpha)}$. and $\kold=\tkold\sqrt{-\alpha/\log(\alpha)}$. The dashed line corresponds to the envelope of the ending points for all values of $\kold$, i.e. to the entropies at which the clusters (non-trivial AMP/TAP fixed points) disappear. In particular, we observe that the clusters with $\tkold = \tilde \kappa_{\rm SAT}=0$ disappear first, thus marking a threshold of $\tknew_{\rm entr} \approx 1.429$ above which the so-called "wide-flat-minima" of \cite{baldassi2016unreasonable,baldassi2021unveiling} exist.}
    \label{fig:different_kappa0}
\end{figure}

\begin{figure}[h]
    \centering
    \includegraphics[width=0.98\textwidth]{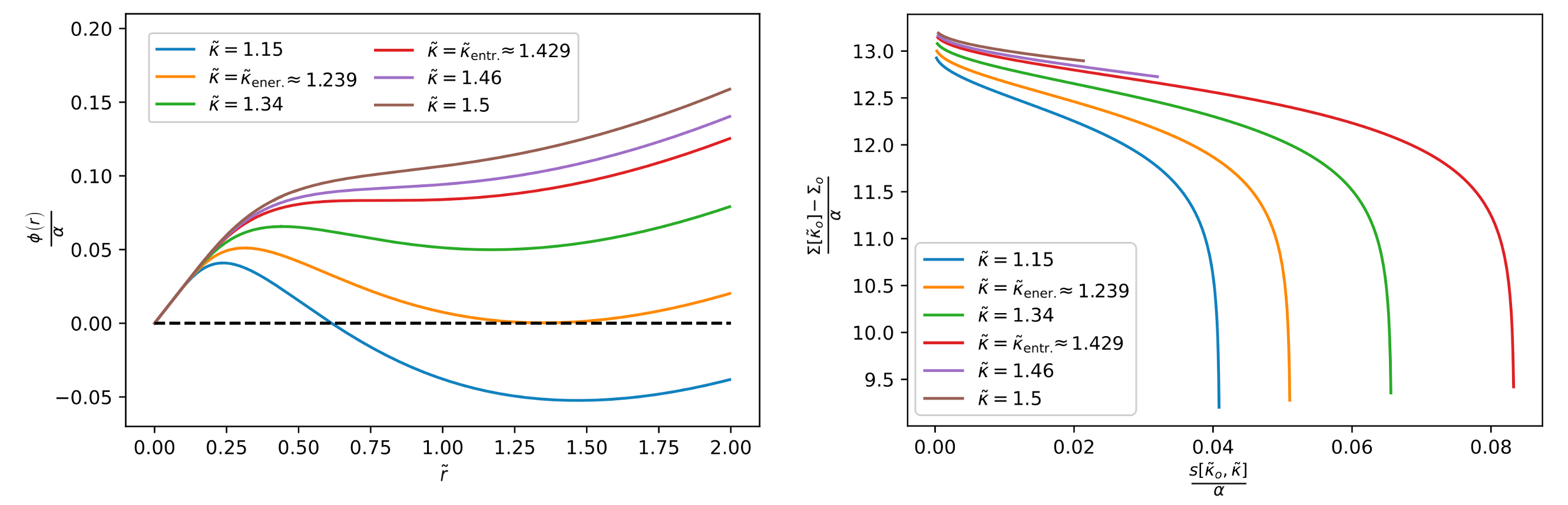}
     \caption{On the left panel we plot the local entropy $\phi_1(r) = h(1-2r) + \alpha \varphi_1(1-2r)$ in the small $\alpha$ limit as a function of the rescaled distance $\tilde{r}=-4r\log(\alpha)/\alpha$; see Section~\ref{sec:maindef}. 
     Each curve corresponds to a different value of $\knew$ (with $\tknew = \knew/\sqrt{-\alpha/\log(\alpha)}\,$), where we have set  $\kold=\kappa_{{\rm SAT}}(\alpha)$ and thus $\kold/\knew{\rightarrow}0$ with $\alpha \rightarrow 0$.  
     First, we highlight the value $\tkenerg$ above which the entropy remains positive for any distance $\tilde{r}>0$. Then, we have a second critical value for $\tknew$ that we label $\tkentro$. It corresponds to the minimum value of $\tknew$ for which the local entropy has a local maximum with $\tilde{r}\neq 0$. On the right panel we plot the complexity $\Sigma[\tkold]$ as a function of the local entropy $s[\tkold,\tknew]$ (with $\tkold=\kold/\sqrt{-\alpha/\log(\alpha)}\,$).  
 In this context, the complexity corresponds to the exponential number of possible planted configurations at $\kold$. The entropy corresponds to the local entropy $\phi_1(r)$ evaluated at its local maxima in $\tilde{r}$ for fixed $\tkold$ and $\tknew$. To obtain a $O(1)$ scaling of the complexity in the low $\alpha$ limit we subtracted $\Sigma_o=\log(2)+\alpha\log(-\alpha/\log(\alpha))$.  Finally, for {  $\tknew>\tkentro$} a span of value for $\tkold$ yields no local maxima for the local entropy. This explains the sudden stop in the right-hand tail of the purple and brown curves (for which  {  $\tknew>\tkentro$}. }
    \label{fig: local entropy and complexity}
\end{figure}

Motivated by these findings, we then study the local entropy of solutions that are at a Hamming distance $Nr$ from the solution planted at $\kold =\kappa_{\sSAT}(\alpha)$. This is akin to the Franz-Parisi potential as studied in the physics of spin glasses~\cite{franz1995recipes}. Here, we compute this potential around a typical solution at $\kold$. Our findings, again in the regime of small $\alpha$, are summarized in Fig.~\ref{fig: local entropy and complexity} (left), where it is apparent that the local entropy as a function of the distance $r$ from a reference solution is monotone when $\kappa \ge \tkentro \sqrt{-\alpha/\log(\alpha)}$ and has a local maximum at an intermediate distance $r_0$ when $\kappa < \tkentro \sqrt{-\alpha/\log(\alpha)}$, with $\tkentro \approx 1.429$ given by implicit equations (\ref{eq: entro transition 1}) and (\ref{eq: entro transition 2}). We also show that no solutions can be found in an interval of distances from the reference solution when $\kappa < \tkenerg \sqrt{-\alpha/\log(\alpha)}$ with $ \tkenerg  \approx 1.239$ 
 given by the implicit equations (\ref{eq: energ transition 1}) and (\ref{eq: energ transition 2}). 

From these results, we note the existence of a logarithmic gap in $1/\alpha$ in the value of $\knew$ where the local entropy curve becomes monotone and the value where the Bansal-Spencer algorithm is proved to succeed, in the regime of small $\alpha$.
It is an interesting open problem to close this gap, either by showing that efficient algorithms can find solutions for all  $\kappa \ge \tkentro \sqrt{-\alpha/\log(\alpha)}$ or by showing the local entropy approach is not indicative of algorithmic hardness.

 \paragraph{The 1RSB computation of the complexity curve:}

We note that in the statistical physics literature, 
clusters as defined above are also associated with a fixed point of the approximate message passing (AMP) algorithm or equivalently the Thouless-Anderson-Palmer (TAP) equations. 
 The cluster entropy can be thus computed as the Bethe entropy corresponding to the AMP/TAP fixed point that is reached by AMP run at $\knew$ and initialized in one of the typical solutions at margin $\kold$. 
For $\knew > \knew_{\rm entr}(\kold)$ the AMP/TAP converges to the same fixed point as would be reached from a random initialization, corresponding to an entropy covering the whole space of solutions. 
Using this relation, the onset of a region where algorithms may be able to find these solutions is then related to 
the existence of solutions such that a AMP/TAP iteration initialized at these points converges to the same fixed point as if the iteration was initialized uniformly at random from $S(\bG,\kappa)$. Indeed, it was observed empirically that solutions found by efficient algorithms always have such a property of AMP/TAP or the belief propagation algorithm converging to the same fixed point as from a random initialization \cite{maneva2007new,braunstein2004survey}.

In the existing statistical physics literature, using the replica method on the one-step replica symmetry breaking level, researchers so far have not found clusters of solutions of extensive entropy in the binary perceptron.  
 This is a point of concern as this method is supposed to count all clusters of solutions corresponding to the TAP/AMP fixed points, including the rare non-equilibrium ones~\cite{monasson1995structural,mezard2003cavity,krzakala2007gibbs,zdeborova2007phase,Semerjian}. 
 This a priori casts doubt on the efficacy of the replica method and the validity of its predictions for the number of clusters of a given size, since the method misses a large part of the phase space (unless some explicit conditioning is done as in \cite{baldassi2015subdominant,baldassi2016local}.)

 We propose, based on the replica method, that the answer to this question lies in the properties of the complexity (the logarithm of the number of clusters) versus entropy (the logarithm of the number of solutions in the clusters) $\Sigma(s)$. We observe that the numerical value of the complexity is rather large compared to the entropy. The slope of $\Sigma(s)$ gives the value of the so-called Parisi parameter $x$ that is therefore rather large: $x \gg 1$. Since the value of $x$ describing the equilibrium properties of the system is always between $0$ and $1$ it is not that surprising that the literature has not investigated solutions of the replica equations corresponding to $x \gg 1$. When we consider a large range of values of $x$ in the standard 1RSB equations for SBP \cite{aubin2019storage}, we obtain the  $\Sigma(s)$ depicted in Fig.~\ref{fig: local entropy and complexity} right. We then provide an argument that leads us to conjecture that in the small $\alpha$ limit, the curve $\Sigma(s)$ corresponds to the one we obtain via the approach of planting at $\kold$. Thus even though, in general, by planting we construct only some of the rare clusters, it seems that in fact we construct the most frequent ones in the limit of small $\alpha$. 

Another property that we unveil is related to the fact that the curve $\Sigma(s)$ is usually expected to be concave.
The non-concave parts were so far considered "unphysical" in the literature (e.g. Fig.\ 8 in \cite{braunstein2003polynomial} or Fig.\ 5 in \cite{zdeborova2007phase}). We show in our present work that the so-called "unphysical branch" of the replica/cavity prediction is actually not  "unphysical" in the SBP and that it reproduces the curve $\Sigma(s)$ obtained from the local entropy calculation at small $\alpha$ and small internal cluster entropy. Moreover, we show that some of the relevant parts of the curve $\Sigma(s)$ cannot be obtained in the usually iterative way of solving the 1RSB equations at a fixed value of 
the Parisi parameter $x$. To access this part of the curve we need to adjust the value of $x$ adaptively in every step when solving the 1RSB fixed point equations iteratively.

\subsection{Organization of the paper and the level of rigour}

The rest of the paper is organized as follows: Section \ref{sec:maindef} defines the local entropy and states the main Theorem \ref{thm:main1} in the small $\alpha$ limit. Section \ref{sec:planted} introduces the planted model and its contiguity to the original model; a key element of the proof. Section \ref{sec:moments_in_planted} contains the moment computations in the planted model, ending with the proof of Theorem \ref{thm:main1}. In Section \ref{sec: local entropy analysis} we use the result of Theorem \ref{thm:main1} and study the properties of the asymptotic formula of the local entropy in the small $\alpha$ limit. In section \ref{sec:1RSB} we study the one-step-replica-symmetry breaking solution of the SBP and its relation to the local entropy. This section investigates general values of $\alpha$, not only the small $\alpha$ limit. Finally, we conclude in section \ref{sec:conclusion}.

Sections \ref{sec:maindef} to \ref{sec:moments_in_planted} are fully mathematically rigorous. In Section \ref{sec: local entropy analysis} we analyze the resulting local entropy formula heuristically, solving the corresponding fixed point equations numerically, and deriving the numerical values for the energetic and the entropic thresholds. In Section \ref{sec:1RSB} we rely on the replica method which is well-accepted and widely used in theoretical statistical physics but not rigorously justified from the mathematical standpoint.   

\section{Definitions and main theorem}
\label{sec:maindef}
In this paper, the local entropy is defined around a solution satisfying the SBP inequalities~\eqref{eq:perceptron} with a \emph{stricter margin} $\kold$. More precisely, for $\kold \le \knew$, let $\bx_0 \in S(\bG, \kold)$, and
 let $Z(\bx_0 , \knew , r)$ be the set of solutions $\by\in S(\bG, \knew)$ which are at Hamming distance $N r$ form $\bx_0$:
\begin{equation}\label{eq:local_Z}
Z(\bx_0 , \knew , r) := \Big|\Big\{\by\in S(\bG, \knew) ~:~ d_H(\bx_0,\by) = Nr \Big\}\Big|\, .
\end{equation}
We then define the local entropy function as the (truncated) logarithm of $Z$ averaged over the choice of $\bx_0$ and the disorder $\bG$:
\begin{equation}\label{eq:local_entropy}
\phi_{N,\delta}(r) := \frac{1}{N} \E_{\bG}\left[ \,  \frac{1}{\big|S(\bG, \kold)\big|} \sum_{\bx_0 \in S(\bG, \kold)} \log_{N\delta} Z(\bx_0, \knew, r) \,\,\Big|\,\,S(\bG, \kold) \neq \emptyset \right]\, , 
\end{equation} 
where $\log_{N\delta}(x) = \max\{\log (x),N \delta\}$, $\delta>0$. This truncation to the logarithm is technically convenient, following~\cite{talagrand2011mean1,nakajima2023sharp}.     
Note that for $\kold= \knew$, the fact that there are no solutions at a distance less than $r_0N$ around $\bx_0$ for some $r_0 = r_0(\knew,\alpha)$ with high probability~\cite{perkins2021frozen} implies that $\phi_{N,\delta}(r) = \delta + o_N(1)$ for all $r < r_0$, and so $\lim_{\delta \to 0} \lim_{N\to \infty} \phi_{N,\delta}(r) = 0$ for $r<r_0$. However, as we increase $\knew$ starting from $\kold$, new nearby solutions are expected to emerge. These are the solutions which are counted by $\phi_{N,\delta}(r)$. This, of course, does not contradict the frozen-1RSB property of $S(\bG,\knew)$ since $\bx_0$ is not typical in $S(\bG,\knew)$.    

{  We show that under a certain concentration condition, Assumption~\ref{assump:concentration} stated in Section~\ref{sec:planted}}, the local entropy $\phi_{N,\delta}(r)$ is given in the limit $N \to \infty$ followed by $\alpha \to 0$ then $\delta \to 0$ by a simple formula which corresponds to the first moment bound (i.e., annealed entropy) in the corresponding planted model of the SBP.  We define binary entropy function
\begin{equation}
    h(m) = -\Big(\frac{1+m}{2}\Big) \log \Big(\frac{1+m}{2}\Big) - \Big(\frac{1-m}{2}\Big) \log \Big(\frac{1-m}{2}\Big)\,,~~~~ m \in (-1,1)\, ,
\end{equation}
and  
\begin{equation}\label{eq:phi1}
\varphi_1(m) = \E \log\P\Big( \big| m Z_0 + \sqrt{1-m^2}Z \big| \le \knew  \, \big| \, Z_0 \Big) \,,~~~~ m \in (-1,1) \, ,
\end{equation}
{ where the outer expectation is taken with respect to $Z_0 \sim N(0,1)$ conditioned on the event $|Z_0| \le \kold$,}  and $Z \sim N(0,1)$ independently of $Z_0$. ($Z_0$ has p.d.f.\ $f$; Eq.~\eqref{eq:f_cond}.) 

\begin{theorem}\label{thm:main1}
For $m \in (-1,1)$, $r = (1-m)/2$ and any sequence $\knew = \knew(\alpha) \to 0$ as $\alpha \to 0$ such that $\alpha \log(1/\knew) \to 0$, under Assumption~\ref{assump:concentration} we have 
\begin{align} \label{eq:limit}
 \lim_{\delta \to 0}\,\limsup_{\alpha \to 0}\, \limsup_{N \to \infty} \Big|\phi_{N,\delta}(r) - \max\big\{h(m)  + \alpha \varphi_1(m) \, , \, \delta\big\}\Big|  = 0 \, .
\end{align}
\end{theorem}
\begin{remark} 
Observe that for $\alpha$ small, we have $\knew_{\sSAT}(\alpha) = \Theta(2^{-1/\alpha})$, therefore the condition on $\kappa$ and $\alpha$ in the theorem can be interpreted as $\knew \gg \knew_{\sSAT}(\alpha)$.
\end{remark}
The proof of Theorem~\ref{thm:main1} can be found in Section~\ref{sec:moments_in_planted}.

\section{The planted model and contiguity}
\label{sec:planted}
The analysis of the local entropy is achieved via a planted model where $\bx_0$ is drawn uniformly at random from the hypercube $\{-1,+1\}^N$ and then the constraint vectors $\bg_a$ are drawn from the Gaussian distribution conditional on $\bx_0$ being a satisfying configuration, i.e., conditional on $\bx_0 \in S(\bG,\kold)$.    

More precisely, we fix the reference (planted) vector $\bx_0\in \{-1,+1\}^N$ and for each $a \in \{1,\cdots,M\}$ we independently draw Gaussian random vectors $\bg_a$ conditioned on the event that  
\begin{equation}\label{eq:planted0}
\big| \langle \bg_a , \bx_0 \rangle \big| \le \kold \sqrt{N} \, .
\end{equation} 
Equivalently, we can write 
\begin{equation}\label{eq:planted}
\bg_a = \frac{1}{\sqrt{N}} w_a\, \bx_0 + \Big(\bI - \frac{1}{N}\bx_0 \bx_0^\top\Big)\, \tilde{\bg}_a \, ,
\end{equation}
where $(\tilde{\bg}_a)_{a=1}^M$ are independent $N(0, \bI_N)$ random vectors and $\bw = (w_a)_{a=1}^M$ has mutually independent coordinates, independent of $(\tilde{\bg}_a)_{a=1}^M$, and distributed as $N(0,1)$ r.v.'s conditioned to be smaller than $\kold$ in absolute value, i.e., they have a p.d.f.\   
\begin{equation}\label{eq:f_cond}
f(w) : = \Big(\frac{1}{\sqrt{2\pi}} e^{-w^2/2} \,  \one\{|w| \le \kold\}\Big) \big/ \P(|Z| \le \kold)\, .
\end{equation}

We let $\P_{\pl}$ be the distribution of the pair $(\bG,\bx_0)$ as per the description above, Eq.~\eqref{eq:planted0}, and $\P_{\rd}$ be their distribution according to the original model where $\bG \in \R^{M \times N}$ is an array of standard Gaussian vectors and $\bx_0$ is drawn uniformly at random from $S(\bG,\knew)$, conditional on the latter being non-empty. We denote by $\E_{\pl}$ and $\E_{\rd}$ the associated expectations.
A simple computation reveals that the ratio of $\P_{\pl}$ to $\P_{\rd}$ is given by   
\begin{equation}
\frac{\rmd \P_{\pl}}{\rmd \P_{\rd}}(\bG, \bx_0) = \frac{\big|S(\bG,\kold)\big|}{\E\big|S(\bG,\kold)\big|} \one\big\{\bx_0 \in S(\bG,\kold)\big\}\,,~~~~ \forall\, \bG \in \R^{M \times N} \, ,\, \bx_0 \in \{-1,+1\}^N \, .
\end{equation}

It was shown in~\cite{abbe2022binary} in the case of binary disorder  where $\bG$ has independent Rademacher entries that the above likelihood ratio has constant order log-normal fluctuations for all $\kold > \kappa_{\sSAT}(\alpha)$; see similar result for Gaussian disorder in~\cite{sah2023distribution} for $\kappa_0$ close to $\kappa_{\sSAT}(\kappa)$.; Tthis implies in particular that $\P_{\rd}$ and $\P_{\pl}$ are mutually contiguous, meaning that for any sequence of events $E_n$ (in the common probability space of $\P_{\rd}$ and $\P_{\pl}$), $\P_{\rd}(E_n) \to 0$ if and only if $\P_{\pl}(E_n) \to 0$, see for instance~\cite[Lemma 6.4]{van1998asymptotic}. In other words, any high-probability event under the planted distribution $\P_{\pl}$ is also a high-probability event under the original distribution $\P_{\rd}$. 
Contiguity allow to compute the local entropy in the planted model, where $\bx_0$ is uniformly distributed over $\{-1,+1\}^N$ instead of $S(\bG,\kold)$, and then transfer the result of this computation to the original model.
{ In our case a result slightly weaker than contiguity is sufficient:  Perkins and Xu~\cite{perkins2021frozen} showed that under a certain numerical assumption (see Assumption 1 therein), 
\begin{equation}\label{eq:small_fluctutations}
\lim_{N \to \infty} \frac{1}{N}\log \Big(\big|S(\bG, \kappa)\big| \big/ \E\big|S(\bG,\kappa)\big|\Big) = 0 \, ,
\end{equation}
in $\P_{\rd}$-probability for all $\kappa > \kappa_{\sSAT}(\alpha)$. As observed in~\cite{achlioptas2008algorithmic,perkins2021frozen} this implies the weaker statement that events of probability $e^{-c n}$, $c>0$ under $\P_{\pl}$ are of probability $o_N(1)$ under $\P_{\rd}$. This turns out to be sufficient for our purposes. This argument is used to prove Lemma~\ref{lem:closeness} below.}    

In addition to the above, we require a concentration property of the restricted partition function $Z(\bx_0,\knew,r)$ with respect to the disorder $\bG$, which we state in more general form as follows: 
 Let $a_j<b_j$, $1 \le j \le M$ be two sequences of real numbers, let $m \in [-1,1]$ and consider the partition function
\begin{equation}
    Z_{N} = \Big|\Big\{ \bx \in \{-1,+1\}^N \, :\, \sum_{i=1}^N x_i = Nm \,,~\langle \bg_j , \bx - m \mathbf{1} \rangle / \sqrt{N} \in [a_j,b_j]~~~\forall 1\le j \le M \Big\}\Big| \, ,
\end{equation} 
where $\bg_j$ are i.i.d.\ standard Gaussian random vectors in $\R^N$.  
\begin{assumption}\label{assump:concentration}
    For any $\delta>0$, $m \in [-1,1]$ and sequences $(a_j), (b_j)$ as above, there exist a constant $C>0$ depending only on $\delta$ and $\Delta := max_j (b_j-a_j)$ such that for all $t>0$,
    \begin{equation}\label{eq:conc}
        \P\Big( \Big|\log_{N\delta} Z_{N} - \E\log_{N\delta} Z_{N}\Big| \ge CN  t\Big) \le \exp\Big(- N \min\{t^2,t\}\Big) \, .
    \end{equation}
\end{assumption}

In models of disordered systems where the free energy is a smooth function of the Gaussian disorder, this concentration follows from general principles of Gaussian concentration of Lipschitz functions, see e.g.~\cite{boucheron2013concentration}. In particular, a stronger version of the above assumption (with no truncation to the logarithm and where the decay on the right-hand side is sub-Gaussian for all $t>0$) holds for the SK and $p$-spin models at any positive temperature, and for the family of \emph{$U$-perceptrons} where the activation function $U$ is positive and differentiable with bounded derivative.  
However, in our case the hard constraints defining the model make concentration far less obvious. Currently, exponential concentration of the truncated log-partition function is known for the half-space model i.e., the one-sided perceptron~\cite{talagrand2011mean1}, and for the more general family of $U$-perceptrons which includes the SBP model under study here, albeit with a non-optimal exponent in $N$ on the right-hand side of Eq.~\eqref{eq:conc}, and with an additional slowly vanishing term on the right-hand side; see~\cite[Proposition 4.5]{nakajima2023sharp}. (The latter paper also studies concentration and the sharp-threshold phenomenon for more general disorder distributions.) For our purposes, an essential feature is exponential decay in $N\theta(t)$ where $\theta: \R_{+} \to \R_{+}$ is any increasing function with $\theta(0)=0$. We assume $\theta(t) = \min\{t^2,t\}$ in the above since this is the sub-exponential tail which is expected, but this is not crucial to the proof.    
Establishing the above assumption is an interesting mathematical problem on its own and goes beyond the scope of this paper.

In the planted model, the local entropy takes the simplified form 
\begin{equation}\label{eq:local_entropy_pl}
\phi^{\pl}_{N,\delta}(r) := \frac{1}{N} \E_{\pl}\left[ \, \log_{N\delta} Z(\bx_0, \knew, r) \right]\, ,
\end{equation} 
where the expectation is with respect to $\bx_0$ taken uniformly in $\{-1,+1\}^N$ and the conditional distribution  $\bG | \bx_0 $ is given by Eq.~\eqref{eq:planted}. 
We now show that under Assumption~\ref{assump:concentration}, $\phi_{N,\delta}(r)$ and $\phi^{\pl}_{N,\delta}(r)$ are close:
\begin{lemma}\label{lem:closeness}
Under Assumption~\ref{assump:concentration} we have for all $r \in (0,1)$,
\begin{equation}
\lim_{N\to \infty} \Big|\phi_{N,\delta}(r)-\phi^{\pl}_{N,\delta}(r)\Big| = 0 \, .
\end{equation}
\end{lemma}
\begin{proof}
We define the random variable $X = (1/N) \log_{N\delta} Z(\bx_0, \knew, r)$. We have $ \E_{\pl}[X] = \phi^{\pl}_{N,\delta}(r)$ and $\E_{\rd}[X] = \phi_{N,\delta}(r)$.  

Now for $t>0$ fixed, we consider the event $A = \big\{\big|X -\E_{\pl}[X] \big| \le t\big\}$.
Under the planted model $\P_{\pl}$ we may assume that $\bx_0 = \one$ by symmetry of the Gaussian distribution. {  Therefore by Assumption~\ref{assump:concentration} (with $\Delta = (1-2r)\knew$) we have $\P_{\pl}(A^c) \le e^{-c N}$, $c = c(t)>0$.  We show that this combined with~\eqref{eq:small_fluctutations} implies $\P_{\rd}(A^c) = o_N(1)$. Indeed for any $\varepsilon>0$,
\begin{align}
\P_{\rd}(A^c) = \E_{\pl}\left[\frac{\E|S(\bG,\kold)|}{|S(\bG,\kold)|} \one_{ A^c}\right] &\le e^{\varepsilon N} \P_{\pl}(A^c) + \P_{\rd}\left(\frac{\E|S(\bG,\kold)|}{|S(\bG,\kold)|} \ge e^{\varepsilon N}\right)\\
&\le e^{(\varepsilon - c)N} + o_N(1)\, ,
\end{align}
where the $o_N(1)$ bound on the second term follows from~\eqref{eq:small_fluctutations}. Taking $\varepsilon = c/2$ shows that $\P_{\rd}(A^c) = o_N(1)$.}
Further, observe that $0 \le X \le \log 2$, $\P_{\rd}$-almost surely. Therefore we have
\begin{equation}\label{eq:diff}
\big|\E_{\rd}[X]-\E_{\pl}[X]\big| \le 
\E_{\rd}\Big[\big|X-\E_{\pl}[X]\big|\Big] 
\le t \P_{\rd}(A) + (2\log 2)\P_{\rd}(A^c) 
\le t+ o_N(1)\, .
\end{equation}
The claim follows by letting $t \to 0$ after $N \to \infty$. 
\end{proof}

 \section{Moment estimates in the planted model}
 \label{sec:moments_in_planted}

Now we aim to calculate the limit of $\phi^{\pl}_N(r)$ as $N\to \infty$ for small $\alpha$. To this end we evaluate the first two moments of $Z(\bx_0, \knew, r)$ and show that the second moment is only larger than the square of the first moment by an exponential factor which shrinks as $\alpha \to 0$. 
Then we show that $\phi^{\pl}_N(r)$ is close to its annealed approximation using Assumption~\ref{assump:concentration}. 

We first need to define two auxiliary functions. For a jointly distributed pair of discrete random variables $(\theta_1,\theta_2)$ let $h(\theta_1,\theta_2)$ be their Shannon entropy. 
For $m,q \in (-1,1)$ we define the function 
 \begin{equation}
 \varphi_2(m,q) = \E \log \P\Big( \big|m Z_0 +Z_1  \big| \le \knew \, , \big| m Z_0 +Z_2  \big| \le \knew \, \big| \, Z_0 \Big)  \, ,
  \end{equation}
where $Z_0 \sim f$ and the pair $(Z_1,Z_2)$ is a centered bivariate Gaussian vector independent of $Z_0$ with covariance 
\begin{equation}\label{eq:cov}
\begin{bmatrix} 1-m^2 & q - m^2 \\ q - m^2 & 1 - m^2\end{bmatrix}\, . 
\end{equation}
\begin{theorem}\label{thm:moments}
Let $\bw = (w_a)_{a=1}^M$ as in Eq.~\eqref{eq:planted}. For $m \in (-1,1)$, $r = (1-m)/2$ we have
\begin{align} 
 \frac{1}{N} &\log \E \Big[Z(\bx_0 , \knew , r) \, \big| \, \bw \Big]\,\, \xrightarrow[N \to \infty]{a.s.} \,\, h(m)  + \alpha \varphi_1(m)\, ,\label{eq:firstmoment}\\
 \mbox{and}~~~~~ 
\frac{1}{N} &\log \E \Big[Z(\bx_0 , \knew , r)^2 \, \big| \, \bw \Big]  \,\,\xrightarrow[N \to \infty]{a.s.} \,\, \max_{q \in [-1,1]} \Big\{ \max_{(\theta_1,\theta_2)} h(\theta_1,\theta_2) 
 + \alpha \varphi_2(m,q) \Big\} \, , \label{eq:2ndmoment}
\end{align}
{  where $\varphi_1$ is defined in Eq.~\eqref{eq:phi1}}, and the inner maximization in Eq.~\eqref{eq:2ndmoment} is over the joint distribution of two $\{-1,+1\}$-valued random variables $(\theta_1,\theta_2)$ such that $\E[\theta_1] =\E[\theta_2] = m$ and $\E[\theta_1\theta_2] =q$.
\end{theorem}
The proof of the above theorem relies on a standard use of Stirling's formula, and is postponed to the end of this section. At this point, if the right-hand side of Eq.~\eqref{eq:2ndmoment} is equal to twice the right-hand side of Eq.~\eqref{eq:firstmoment}, a mild concentration argument would allow us to conclude that $\phi^{\pl}_N(r)$ is given by Eq.~\eqref{eq:firstmoment} in the large $N$ limit. This equality would follow if the value $q=m^2$ is a maximizer in Eq.~\eqref{eq:2ndmoment}. This does not appear to be the case for any values of $\alpha,\kold,\knew$. However, we show that the difference is vanishing when $\alpha \to 0$.            
Let
 \begin{align}
 \phi_1(m) &= h(m)  + \alpha \varphi_1(m) \, ,\\
 \phi_2(m) &= \max_{q \in [-1,1]} \Big\{ \max_{(\theta_1,\theta_2)} h(\theta_1,\theta_2)  + \alpha \varphi_2(m,q) \Big\} \, . 
 \end{align}
\begin{lemma}\label{lem:closeness2}
Assume  $\kold<1$ and $\knew^2 \ge \kold^2/(1-\kold^2)$. Then for all $m \in (-1,1)$,
\begin{equation}
0 \le \phi_2(m) - 2\phi_1(m) \le \alpha \log\big(1/p(\knew)\big) \, , 
\end{equation}
where $p(\knew) = \P\big(|Z| \le\knew\big)$, $Z \sim N(0,1)$. In particular the above difference tends to zero wherever $\alpha \to 0, \knew \to 0$ with $\alpha \log(1/\knew) \to 0$, and $\kappa_0 \ll \kappa$. 
 \end{lemma}

\begin{proof}
We first remark that by sub-additivity of the entropy, 
\begin{equation}\label{eq:subadd}
h(\theta_1,\theta_2) \le 2h(m)\, ,
\end{equation}
with equality  if and only if  the pair $(\theta_1,\theta_2)$ is independent, i.e., if $q = m^2$. Moreover we remark that for all $q \in [-1,1]$, 
\begin{equation}\label{eq:supadd}
2 \varphi_1(m) \le \varphi_2(m,q) \le \varphi_1(m) \, ,
\end{equation}
where the lower bound follows from the Gaussian correlation inequality~\cite{latala2017royen,royen2014simple} (with equality if $(Z_1,Z_2)$ are independent, i.e., $q=m^2$) and the upper bound by Cauchy Schwarz (with equality if $Z_1 = Z_2$, i.e., $q=1$).

Using the bounds~\eqref{eq:subadd} and~\eqref{eq:supadd} we have
\begin{equation}
\phi_2(m) \le 2h(m) + \alpha \varphi_1(m) \, ,
\end{equation}
whence,
\begin{equation}
0 \le \phi_2(m) - 2\phi_1(m) \le -\alpha\varphi_1(m) \, .
\end{equation}
It remains to show that $\varphi_1$ is a non-decreasing function so that $\varphi_1(m) \ge \varphi_1(0) = \log p(\knew)$. A simple computation of the derivative of $\varphi_1$ reveals that
\begin{equation}
\varphi_1'(m) = \E\left[ \frac{a_+'(m) e^{-a_+^2(m)/2} - a_{-}'(m)e^{-a_-^2(m)/2}}{\int_{a_{-}(m)}^{a_+(m)} e^{-u^2/2} \rmd u} \right]\, ,~~\mbox{with}~~~ a_\pm(m) = \frac{-mZ_0 \pm\knew}{\sqrt{1-m^2}} \, ,
\end{equation}
{  where $Z_0$ has p.d.f.~\eqref{eq:f_cond}, and the expectation is taken with respect to $Z_0$.} 
Using $a_{\pm}'(m) = \frac{-Z_0 \pm m\knew}{(1-m^2)^{3/2}}$, the numerator of the above expression can be written as follows:  
\begin{equation}
a_+'(m) e^{-a_+^2(m)/2} - a_{-}'(m)e^{-a_-^2(m)/2} = \frac{1}{(1-m^2)^{3/2}}\Big((m\knew -Z_0) e^{-a_+^2(m)/2} + (m\knew + Z_0)e^{-a_-^2(m)/2}\Big)\, .
\end{equation}
{ We will show that the above display is non-negative for all values of $Z_0$.} First, note that this expression is even as a function of $Z_0$ so we assume $Z_0\ge0$ without loss of generality. Now, since $Z_0 \le \kold$ a.s.\ the above expression is nonnegative if $m \ge \kold/\knew$. Now let us consider the remaining case $m < \kold/\knew$. Processing the numerator further we obtain  
\begin{align}
&a_+'(m) e^{-a_+^2(m)/2} - a_{-}'(m)e^{-a_-^2(m)/2}\\
&=\frac{e^{-(\knew^2+m^2Z_0^2)/(2(1-m^2))}}{(1-m^2)^{3/2}} \Big((m\knew -Z_0) e^{m\knew Z_0 / (1-m^2)} + (m\knew + Z_0)e^{-m\knew Z_0 / (1-m^2)}\Big)\\
&=\frac{e^{-(\knew^2+m^2Z_0^2)/(2(1-m^2))} \cosh\big(m\knew Z_0 / (1-m^2)\big)}{(1-m^2)^{3/2}} \Big(m\knew - Z_0 \tanh\big(m\knew Z_0 / (1-m^2)\big)\Big)\, .
\end{align}
From the bound $\tanh(x) \le x$ for $x \ge 0$ we see that 
\begin{equation}
m\knew - Z_0 \tanh\big(m\knew Z_0 / (1-m^2)\big) \ge m\knew\Big(1-\frac{Z_0^2}{1-m^2}\Big)\, .
\end{equation}
This is non-negative as long as $1 - \kold^2/(1-m^2) \ge 0$. Since $m < \kold/\knew$, this is verified when $1-(\kold/\knew)^2 \ge \kold^2$, i.e., when $\knew^2\ge \kold^2/(1-\kold^2)$.
\end{proof}

Next, we are ready to prove the main result of this section:
\begin{theorem}\label{thm:main_planted}
Under the assumptions of Theorem~\ref{thm:main1} we have 
\begin{align} 
 \lim_{\delta \to 0}\,\limsup_{\alpha \to 0}\, \limsup_{N \to \infty} \Big|\phi_{N,\delta}^{\pl}(r) - \max\big\{\phi_1(m)\, , \, \delta\big\}\Big|  = 0 \, ,\quad \quad r = (1-m)/2\, ,
\end{align}
where the limit in $\alpha$ is such that $\alpha \to 0, \knew \to 0$ with $\alpha \log(1/\knew) \to 0$.
\end{theorem}
We see that Theorem~\ref{thm:main1} follows from  Theorem~\ref{thm:main_planted} and Lemma~\ref{lem:closeness}. Now we prove Theorem~\ref{thm:main_planted}:
\begin{proof}
We write $Z=Z(\bx_0, \knew, r)$. All probabilities and expectations are taken under $\P_{\pl}$. 
For fixed $t,t'>0$ to be chosen later we define the events 
\begin{align}
&A = \left\{ \frac{1}{N}\log_{N\delta}\E_{\pl}\big[Z \,|\, \bw\big] - \frac{1}{N}\log_{N\delta} Z  \le \frac{\log 2}{N} \right\} \, ,~~~~ B = \left\{  \frac{1}{N} \log_{N\delta} Z -\phi_{N,\delta}^{\pl}(r) \le t' \right\} \, ,\\
C &= \left\{ \max\big\{\phi_1(m)\, , \, \delta\big\} - \frac{1}{N}\log_{N\delta}\E_{\pl}\big[Z \,|\, \bw\big]  \le t \right\}  \, ,~~~\mbox{and}~~
D = \left\{\frac{1}{N}\log\E_{\pl}\big[Z^2 \,|\, \bw\big]  - \phi_2(m) \le t \right\} \, .
\end{align}
First, we note that by Jensen's inequality, 
\begin{align}
\frac{1}{N}\E_{\pl} \big[\log_{N\delta} Z \,|\, \bw\big]
&\le \frac{1}{N} \log \E_{\pl}\big[ \max\{e^{N\delta}, Z\} \,|\, \bw\big]\\
&\le \frac{1}{N} \log \Big(e^{N\delta}+\E_{\pl}\big[Z \,|\, \bw \big]\Big)\\
&\le \frac{1}{N} \log \Big(2\max\big\{e^{N\delta}, \E_{\pl}\big[Z \,|\, \bw \big] \big\}\Big)\\
&= \frac{\log 2}{N} + \frac{1}{N}\log_{N\delta} \E_{\pl}\big[Z \,|\, \bw \big] \, .\label{eq:ine_1}
\end{align}
Since by Theorem~\ref{thm:moments}, $\frac{1}{N}\log \E_{\pl}\big[Z \,|\, \bw \big] \to \phi_1(m)$ almost surely as $N\to \infty$, we have by dominated convergence, 
\begin{equation}
\limsup_{N\to\infty} \phi_{N,\delta}^{\pl}(r) \le  \max\big\{\phi_1(m)\, , \, \delta\big\} \, .
\end{equation}
Next, under $A \cap B \cap C$ we have
\begin{equation}\label{eq:bound0}
\max\big\{\phi_1(m)\, , \, \delta\big\}  - \phi_{N,\delta}^{\pl}(r) \le \frac{\log 2}{N} + t+t' \, .
\end{equation}
Now the goal is to show that $\P_{\pl}(A\cap B \cap C) >0$.  
Let $\bw$ be such that $C \cap D$ holds. It follows by the Paley-Zigmund inequality that
\begin{align}
    \P_{\pl}(A \,|\, \bw ) &= \P_{\pl} \Big( \max\big\{e^{N\delta}, Z\big\} \ge \max\big\{e^{N\delta},\E_{\pl}[Z\, |\, \bw]\big\}/2  \, \big|\, \bw\Big) \\
    &\ge \P_{\pl} \Big( Z \ge \max\big\{e^{N\delta},\E_{\pl}[Z\, |\, \bw]\big\}/2 \, \big| \, \bw \Big)\\
    &\ge  \frac{\max\big\{\E_{\pl}[Z\, |\, \bw]^2 , e^{2N\delta}\big\}}{4\E_{\pl}[Z^2\, |\, \bw]} \\
    &\ge \frac{1}{4}\frac{\E_{\pl}[Z\, |\, \bw]^2}{\E_{\pl}[Z^2\, |\, \bw]} \\
    &\ge \frac{1}{4} \exp\Big(-N\big(\phi_2(m)-2\phi_1(m)+3t+2\delta\big)\Big) \, ,
\end{align}
where the last inequality follows from $C \cap D$. 
From Lemma~\ref{lem:closeness2} we have $\phi_2(m)-2\phi_1(m) \le \alpha \log(1/p(\knew))$ when $\knew^2 \ge \kold^2(1-\kold^2)$. Next by Theorem~\ref{thm:moments} we have $\P_{\pl}(C \cap D) \ge 1/2$ for $N$ large enough (it is actually $1-o_N(1)$). It follows that 
\begin{equation}
\P_{\pl}(A \cap C \cap D) \ge \frac{1}{8} \exp\Big(-N\big(\alpha \log(1/p(\knew))+3t+2\delta\big)\Big)\, . 
\end{equation}

On the other hand, by our concentration Assumption~\ref{assump:concentration}, $\P_{\pl}(B^c) \le \exp(-N \min\{t'^2/K^2,t'/K\})$ where $K= K(\delta,\Delta)$, $\Delta = (1-2r)\knew$ is the constant appearing in the assumption, we have by a union bound
\begin{equation}
\P_{\pl}( A \cap B \cap C \cap D) \ge \P_{\pl}(A \cap C \cap D) - \P_{\pl}(B^c) \ge  \frac{1}{8}\exp\Big(-N\big(\alpha \log(1/p(\knew))+3t+2\delta\big)\Big) - \exp\big(-N \min\{t'^2/K^2,t'/K\}\big) \, .
\end{equation}
Now we choose $t = \frac{1}{3}(2\delta+\alpha \log(1/p(\knew)))$ and $t'=t_N'$ such that $\min\{t_N'^2/K^2,t_N'/K\} = (\log 16) /N + 2\alpha \log(1/p(\knew)) + 4\delta$. 
We obtain
\begin{equation}
\P(  A \cap B \cap C \cap D) \ge \frac{1}{16}\exp\big(-2N\alpha \log(1/p(\knew)) - 4N\delta\big) > 0\, .
\end{equation}
 Therefore the bound~\eqref{eq:bound0} holds with this choice of parameters:
\begin{equation}\label{eq:bound1}
 \max\big\{\phi_1(m) \, , \, \delta\big\}  - \phi_{N,\delta}^{\pl}(r) \le \frac{\log 2}{N} +\frac{1}{3}\alpha \log(1/p(\knew)) + \frac{2}{3} \delta + t_N'  \, ,
\end{equation}
and we obtain 
\begin{equation}\label{eq:bound2}
\liminf_{N \to \infty} \phi_{N,\delta}^{\pl}(r) \ge  \max\big\{\phi_1(m) \, , \, \delta\big\}   -\frac{1}{3}\alpha \log(1/p(\knew)) - \frac{2}{3} \delta - t_{\infty}'  \, .
\end{equation}
Letting $\alpha \to 0,\kappa \to 0$ such that $\alpha \log(1/\kappa)\to0$ and then $\delta \to 0$  concludes the proof.
\end{proof}

\begin{proof}[Proof of Theorem~\ref{thm:moments}]
Let us start with the first moment. First, we have
\begin{align}
\E \Big[Z(\bx_0, \knew , r) \, \big| \, \bw \Big] &= \sum_{\bx \in \{-1,+1\}^N} \E \Big[\one\big\{\bx\in S(\bG, \knew)\, , \langle \bx_0 , \bx \rangle = N m \big\} \, \big| \, \bw \Big] \\
&= \frac{1}{2^N}\sum_{\bx_0,\bx \in \{-1,+1\}^N}   \one\big\{ \langle \bx_0 , \bx \rangle = N m \big\} \prod_{a=1}^{M}\P\Big( \big|\langle \bg_a , \bx \rangle  \big| \le \knew \sqrt{N} \, \big| \, \bx_0, w_a\Big) \, .
\end{align}
We further have 
\begin{align}
\frac{1}{\sqrt{N}}\langle \bg_a , \bx \rangle &= \frac{w_a}{N} \langle \bx , \bx_0 \rangle  + \frac{1}{\sqrt{N}}\langle \tilde{\bg}_a,  \big(\bI - \frac{1}{N}\bx_0 \bx_0^\top\big) \bx \rangle  \\
& \stackrel{\rmd}{=} m w_a + \sqrt{1-m^2} Z\, ,
\end{align}
where $Z \sim N(0,1)$ independently.
Using Stirling's formula, we obtain 
\begin{align}
\frac{1}{N} \log \E \Big[Z(\bx_0 , \knew , m) \, \big| \, \bw \Big] &=  h(m) + \frac{1}{N} \sum_{a=1}^{M} \log \P\big( | m w_a + \sqrt{1-m^2}Z | \le \knew  \, \big| \, w_a \big) + o_N(1) \, .
\end{align}
An application of the strong law of large numbers yields the formula in Eq.~\eqref{eq:firstmoment}.
 
We now calculate the second moment:
\begin{align}
\E \Big[Z(\bx_0 , \knew , m)^2 \, \big| \, \bw \Big] &= \sum_{\bx^1, \bx^2 \in \{-1,+1\}^N} \E \Big[\one\big\{\bx^i \in S(\bG, \knew)\, , \langle \bx_0 , \bx^i \rangle = N m \, , i=1,2 \big\}\Big] \\
&= \frac{1}{2^N}\sum_{\bx_0,\bx^1, \bx^2 \in \{-1,+1\}^N}   \hspace{-0.2cm}\one\big\{ \langle \bx_0 , \bx^i \rangle = N m \, , i=1,2\big\} \prod_{a=1}^{M}\P\Big( \big|\langle \bg_a , \bx^i \rangle  \big| \le \knew \sqrt{N} \, , i=1,2 \, \big| \, \bx_0, w_a\Big) \, .
\end{align}
Fix $m,q \in [-1,1]$ and three vectors $\bx_0, \bx^1$, $\bx^2$ such that $\langle\bx^1,\bx^2\rangle = N q$, and  $\langle\bx^i,\bx_0\rangle = N m$, for $i=1,2$. Then as before,
\begin{align}
\frac{1}{\sqrt{N}}\langle \bg_a , \bx^i \rangle &= \frac{w_a}{N} \langle \bx^i , \bx_0 \rangle  + \frac{1}{\sqrt{N}}\langle \tilde{\bg}_a,  \big(\bI - \frac{1}{N}\bx_0 \bx_0^\top\big) \bx^i \rangle  \\
& \stackrel{\rmd}{=} m w_a +  Z_i\, ,
\end{align}
where the pair $(Z_1,Z_2)$ is defined as in Eq.~\eqref{eq:cov}. 
 Furthermore, by symmetry we can assume that $\bx_0 = \one$, and we define the set 
 \begin{equation}
 C(m,q) = \Big\{\bx^1, \bx^2 \in \{-1,+1\}^N\, :\,  \langle\bx^1,\bx^2\rangle = N q \, , \langle \one , \bx^i \rangle = N m \, , i=1,2\Big\} \, .
 \end{equation}
We have
\begin{align}
\E \Big[Z(\bx_0 , \knew , m)^2 \, \big| \, \bw \Big] &= \sum_{q \in [-1,1] \cap \Z/N}  \big|C(m,q) \big|
\prod_{a=1}^{M}  \P\Big( \big| m w_a +Z_i  \big| \le \knew \, , i=1,2 \, \big| \, w_a\Big) \, .
\end{align}
Therefore 
\begin{equation}\label{eq:cmq}
\frac{1}{N} \log \E \Big[Z(\bx_0 , \knew , m)^2 \, \big| \, \bw \Big]  = \max_{q \in [-1,1] \cap \Z/N} \left\{ \frac{1}{N} \log \big|C(m,q) \big| 
+  \frac{1}{N} \sum_{a=1}^{M} \log \P\Big( \big| m w_a +Z_i  \big| \le \knew \, , i=1,2 \, \big| \, w_a\Big) \right\} + o_N(1) \, .
\end{equation}

{  Next we compute the size of $C(m,q)$. We can write
 \begin{equation}
\big|C(m,q) \big| = \sum_{\vect{k}} \binom{N}{\vect{k}} ,
 \end{equation}
 where $\binom{N}{\vect{k}} = \binom{N}{k_{+,+},k_{+,-},k_{-,+},k_{-,-}}$ is the multinomial coefficient and the sum is restricted to those integers satisfying    
 \begin{align}
k_{+,+}+k_{+,-}+k_{-,+}+k_{-,-} &= N\,,\\
k_{+,+}+k_{+,-}-k_{-,+}-k_{-,-} &= Nm\,,\\
k_{+,+}-k_{+,-}+k_{-,+}-k_{-,-} &= Nm\,,\\
k_{+,+}-k_{+,-}-k_{-,+}+k_{-,-} &= Nq \, .
\end{align}
}
Using Stirling's formula we find   
 \begin{equation}
\frac{1}{N} \log \big|C(m,q) \big| = \max_{(\theta_1,\theta_2)} h(\theta_1,\theta_2) + o_N(1)\, ,
 \end{equation}
 where the maximization is as in Eq.~\eqref{eq:2ndmoment}. {  (The correspondence being that $\P(\theta_1=\epsilon,\theta_2=\epsilon') = k_{\epsilon,\epsilon'}/N$.)}

Moreover, letting $\theta(w) := \log \P\Big( \big| m w +Z_i  \big| \le \knew \, , i=1,2 \, \big| \, w\Big)$ the average $X_N = \frac{1}{N} \sum_{a=1}^M  \theta(w_i)$ has a subGaussian tail in $N$, i.e., $\P(|X_N - \E[X_N]| \ge t) \le 2 e^{-Nt^2/(2C)}$ for some constant $C>0$ by the Azuma-Hoeffding inequality. Since the maximum in Eq.~\eqref{eq:cmq} is taken over no more than $2N+1$ values we can let $t=t_N \to 0$ slowly with $N$ such that $\sum_{N} N e^{-Nt_N^2/(2C)}<\infty$. The Borel--Cantelli lemma and continuity allow us to conclude the proof. 
\end{proof}

\section{Analysing the local entropy and its thresholds}
\label{sec: local entropy analysis}
Having shown in Theorem \ref{thm:main1} that the local entropy $\phi_{N,\delta}(r)$ is asymptotically given by the formula  $\max\{0,\phi_1(r)\}$ in the limit $N\to \infty$, $\alpha \to 0$ then $\delta \to 0$, where
\begin{align}
     \phi_1\left(r=\frac{1-m}{2}\right) = h(m)  + \alpha \varphi_1(m)\, .
\end{align}
 We will now focus on the analysis of this function. 
In App.~\ref{app: A} we derive the local entropy for generic values of these parameters and show a posteriori how we can recover the limit presented above. 

A first step to simplify our analysis is to rewrite $\varphi_1(m)$ in the following fashion
\begin{align}
    \varphi_1(m)= \int\mathcal{D}Z_0 \frac{\one\{|Z_0| \le \kold \}}{\mathcal{N}_{\kold}} \, \log\left\{\frac{1}{2}\erf{\frac{\knew+Z_0}{\sqrt{2(1-m^2)}}}+\frac{1}{2}\erf{\frac{\knew-Z_0}{\sqrt{2(1-m^2)}}}\right\} \, ,
\end{align}
where $\mathcal{N}_{\kold}=\P(|Z_0|\le \kold)$, and we let $\mathcal{D}Z_0 = \frac{e^{-Z_0^2/2}}{\sqrt{2\pi}}\rmd Z_0$ denote the Gaussian measure. We recall that the error function is
\begin{align}
    \erf{x}=\frac{2}{\sqrt{\pi}}\int_0^x e^{-t^2}dt\, .
\end{align}

In fact, when $\alpha\ll 1$ the local entropy is a non-trivial function for only a restricted range of parameters $\knew$ and~$m$.
For this to happen the entropic and energetic contributions have to be comparable. This leads us to introduce a rescaling of the form 
\begin{align}
1-{m}^2= -\alpha  \tilde{r} /\log(\alpha) \, ,~~~~ \kold={\tkold} \sqrt{-\alpha /\log(\alpha)}\, ,~~~~ \mbox{and}~~~~ \knew={\tknew} \sqrt{-\alpha /\log(\alpha)} \, ,
\end{align}
in order to have both $\varphi_1(m)$ and $h(m)$ contributing as a $\mathcal{O}(\alpha)$ in the local entropy when $\alpha\ll 1$.
This first indicates that we can restrict our analysis to a regime where $1-m\ll 1 $. Consequently, the entropic term is simplified to
\begin{align}
    h(m) &= -\frac{1-m}{2}\log{\left(\frac{1-m}{2}\right)}+o(\alpha)\\
         &=\frac{\alpha\tilde{r}}{4}+o(\alpha)\, .\nonumber 
\end{align}

Then, using this rescaling we obtain the simplified form of the local entropy and the equation for its local maxima (at $\tilde{r}\neq 0$)
\begin{align}
\label{eq: entropy planting rescaled}
    \phi_1\left(r=\frac{-\alpha  \tilde{r} }{4\log(\alpha)}\right)&= \frac{ \alpha\tilde{r}}{4}+\frac{\alpha }{{\mathcal{N}}_{\tkold}}\int\! \mathcal{D}Z_0\, \one\{|Z_0| \le \tkold \}\,\log\left\{\frac{1}{2}\erf{\frac{\tknew+Z_0}{\sqrt{2\tilde{r}}}}+\frac{1}{2}\erf{\frac{\tknew-Z_0}{\sqrt{2\tilde{r}}}}\right\}+o(\alpha)\, ,\\
\label{eq: local max planting rescaled}
   1&=\frac{4}{{\mathcal{N}}_{\tkold}}\int \mathcal{D}Z_0\frac{\one\{|Z_0| \le \kold \}}{\erf{\frac{\tknew+Z_0}{\sqrt{2\tilde{r}}}}+\erf{\frac{\tknew-Z_0}{\sqrt{2\tilde{r}}}}}   \left\{\frac{(\tknew+Z_0)e^{\frac{-(\tknew+Z_0)^2}{2\tilde{r}}}}{\sqrt{2\pi}\tilde{r}^{3/2}}+\frac{(\tknew-Z_0)e^{\frac{-(\tknew-Z_0)^2}{2\tilde{r}}}}{\sqrt{2\pi}\tilde{r}^{3/2}}\right\}+o(\alpha)\, 
\end{align}
with again $\mathcal{N}_{\tkold}=\P(|Z_0|\le \tkold)$.
The presence of this local maximum in the potential tells us that there is a cluster of atypical solutions with margin~$\knew$ around each typical configuration with margin $\kold$. In the following, we will denote as $s[\tkold,\tknew]$ the local entropy evaluated at this maximum.

In Fig.~\ref{fig:different_kappa0} we display the behavior of the local entropy $s[\tkold,\tknew]$ as a function of $\tkold$ and $\tknew$. As outlined by the dashed line, clusters exist only for a finite span of values for $\tknew$, which depends on the margin $\tkold$ of the reference vector $\bx_0$. Defining $\tkentro(\tkold)$ as the critical value of $\tknew$ for which clusters disappear, we see from the figure that $\tkentro\equiv {\rm min}_{\tkold} \tkentro(\tkold) =\tknew_{\rm entr}(\tkold =0)$. In other words, the first clusters to disappear are the ones formed around a reference vector at $\tkold =0$. In particular, this corresponds to planting at $\tkold =\tilde{\kappa}_{\sSAT}$ as we have
\begin{equation}
\kappa_{\sSAT}(\alpha) \underset{\alpha\rightarrow0}{\sim}\sqrt{\frac{\pi}{2}}e^{\frac{-\log 2}{\alpha}} \quad \text{and}\quad\tilde{\kappa}_{\sSAT}=\kappa_{\sSAT}(\alpha)\sqrt{\frac{-\log(\alpha)}{\alpha}}\,\, \xrightarrow[\alpha \to 0]{} \,\, 0\, .
\end{equation}
Since clusters are associated with AMP/TAP fixed points, $\tkentro$ corresponds to the margin above which the AMP/TAP initialized close to a typical solution with margin $\kold={\kappa}_{\sSAT}$ converges to the same fixed point as would be reached from a random initialization. The existence of solutions from which AMP/TAP converges to this trivial fixed point was linked to the onset of a region where algorithms may be able to find solutions. More precisely, numerical evidence in the literature suggests that solutions that are found by efficient algorithms do not correspond to other AMP/TAP fixed points than the one reached from random initialization \cite{maneva2007new,braunstein2004survey}.  

As shown in Fig.~\ref{fig: local entropy and complexity}, in which we plant at $\kold=\kappa_{\sSAT}$, the local entropy undergoes two interesting thresholds with distinctive values of $\tknew$ (for fixed $\alpha$). One being the value of $\tknew$ above which the potential remains positive for all $m$, we will refer to it as the {\it energetic} threshold with $\tknew=\tkenerg(\tkold)$. In other words, this means that above this critical margin we can find solutions to the symmetric binary perceptron with margin~$\knew$ at any distance from the reference vector. The second critical value for $\tknew$ corresponds to the loss of the local maximum at $m\neq 0$ in the potential. This corresponds to the {\it entropic} threshold that we mentioned earlier with $\tknew=\tkentro(\tkold)$.

In the two following sections, we focus our analysis on these two thresholds in the case where $\tkold=0$. Again, this choice is justified by the fact that the {\it energetic} and {\it entropic} threshold happen first when planting at $\tkold=\kappa_{\sSAT} \to_{\alpha \to 0} 0$, i.e.
\begin{align}
    \tkenerg(0)&=\underset{\tkold}{\rm min}\,\tkenerg(\tkold)\, ,\\
    \tkentro(0)&=\underset{\tkold}{\rm min}\,\tkentro(\tkold)\, .
\end{align}
Similarly to the entropic threshold, we will use in the following the shortening $\tkenerg(0)=\tkenerg$.

\subsection{Energetic threshold}
The {\it energetic} threshold occurs when in a range of intermediate distances the local entropy $\phi_1(r)$ is negative. This means that we want to find the exact point where the minimum of the entropy (excluding $m=1$) is zero. 
We start by setting $\tkold=0$ in Eq.~(\ref{eq: entropy planting rescaled}) to obtain the simplified form of the local entropy
\begin{align}
\label{eq: simplified rescaled free energy}
 \phi_1\left(r=\frac{-\alpha  \tilde{r}}{4\log(\alpha)}\right)=\frac{\alpha \tilde{r}}{4}+\alpha\log\left[{\rm erf}\left(\frac{\tkenerg}{\sqrt{2 \tilde{r}}}\right)\right]+o(\alpha)\, .
\end{align}
The potential is then null when
\begin{align}
\label{eq: energ transition 1}
    1=-4\frac{\log\left[{\rm erf}\left(\frac{\tkenerg}{\sqrt{2 \tilde{r}}}\right)\right]}{\tilde{r}}+o(1)
\end{align}
and the r.h.s of the upper equation has a maximum for
\begin{align}
\label{eq: energ transition 2}
    \frac{\log\left[{\rm erf}\left(\frac{\tkenerg}{\sqrt{2 \tilde{r}}}\right)\right]}{\tilde{r}^2}+\sqrt{\frac{2}{\pi}}\frac{\tkenerg e^{\frac{-\tkenerg^2}{2\tilde{r}}}}{{\rm erf}\left(\frac{\tkenerg}{\sqrt{2 \tilde{r}}}\right) \tilde{r}^{5/2}}=o(1)\, .
\end{align}
Finally, if we solve the two previous equations, we obtain the set of values $\{\tkenerg,\tilde{r}\}$ for which the potential stops being negative for any value of the magnetization $m$. Numerically we obtain
\begin{align}
    \kenerg&=\tkenerg \sqrt{-\alpha /\log(\alpha)}+o\left(\frac{\alpha}{\log(\alpha)}\right)\approx1.238518 \sqrt{-\alpha /\log(\alpha)}\, ,\\
    1-m^2&=-\alpha\tilde{r} /\log(\alpha)+o\left(\frac{\alpha}{\log(\alpha)}\right)\approx-1.351180\,\alpha /\log(\alpha)\, .
\end{align}

\subsection{Entropic threshold}
The {\it entropic} threshold occurs when the local maximum other than $m \neq 0$ of the free entropy cease to exist. We recall that the local entropy for $\tkold=0$ reads
\begin{align}
\label{eq: simplified rescaled free energy 2}
 \phi\left(r=\frac{-\alpha  \tilde{r}}{4\log(\alpha)}\right)=\frac{\alpha \tilde{r}}{4}+\alpha\log\left[{\rm erf}\left(\frac{\tkenerg}{\sqrt{2 \tilde{r}}}\right)\right]+o(\alpha)
\end{align}
with a non-trivial local maximum obtained by solving the fixed point equation
\begin{align}
\label{eq: entro transition 1}
    \frac{\alpha}{4}=\alpha\frac{\tkenerg e^{-\frac{\tkenerg^2}{2\tilde{r}}}}{\sqrt{2\pi}  \tilde{r}^{3/2}  {\rm erf}\left(\frac{\tkenerg}{\sqrt{2 \tilde{r}}}\right)}+o(\alpha)\,.
\end{align}
Again, the r.h.s. of the previous has a maximum for
\begin{align}
\label{eq: entro transition 2}
    \frac{\tkenerg^2}{2\tilde{r}^2}-\frac{3}{2\tilde{r}}+\frac{\tkenerg e^{-\frac{\tkenerg^2}{2\tilde{r}}}}{\sqrt{2\pi}\tilde{r}^{3/2}  {\rm erf}\left(\frac{\tkenerg}{\sqrt{2 \tilde{r}}}\right)}=o(1)\, .
\end{align}

Finally, we can solve numerically the two previous equations and we obtain
\begin{align}
    \kentro&=\tkentro \sqrt{-\alpha /\log(\alpha)}+o\left(\frac{\alpha}{\log(\alpha)}\right)\approx1.428754 \sqrt{-\alpha /\log(\alpha)}\, ,\\
    1-m^2&=-\alpha\tilde{r} /\log(\alpha)+o\left(\frac{\alpha}{\log(\alpha)}\right)\approx-0.782487\,\alpha /\log(\alpha)\, .
\end{align}

\subsection{Complexity versus entropy}
\label{subsec: Complexity vs entropy}
In this section, we focus on the relation between the complexity of the clusters around the high-margin solutions and their local entropy.   
We define the complexity as the logarithm of the number of clusters around solutions at margin $\kold$, normalized by $N$, and we recall that the local entropy of a cluster is the value of the local entropy $\phi_1(r=\frac{1-m}{2})$ at the nearest local maximum to the reference solution. 
By contiguity to the planted model, the clusters of solutions with margin $\knew>\kold$ living around two different planted configurations are distant, since the reference configurations are nearly orthogonal with high probability. Thus, heuristically, counting their exponential number (or complexity) simply consists of enumerating the number of typical solutions at $\kold$ we can plant. 

Taking these previous considerations into account the obtained clusters have a complexity that depends solely on $\kold$ while their local entropy is a function of $\knew$ and $\kold$. Fixing $\knew$ while tuning $\kold$ enables us to scan across sets of clusters with different complexities and local entropies, all containing atypical solutions of the symmetric binary perceptron with margin~$\knew$. More specifically, the complexity is
\begin{align}
\label{eq: complexity}
    \Sigma[\kold]=\log 2+\alpha\log \P(|Z| \le \kold)=\log 2 +\alpha\log\left\{\erf{\frac{\kold}{\sqrt{2}}}\right\}\, .
\end{align}
and the entropy of a cluster is 
\begin{align}
   s[\kold,\knew]=&-\frac{1-m}{2}\log{\left(\frac{1-m}{2}\right)}+\frac{\alpha}{\mathcal{N}_{\kold}}\int\! \mathcal{D}Z_0\, \one\{\vert Z_0\vert \le \kold\}\,\log\left\{\frac{1}{2}\erf{\frac{\knew+Z_0}{\sqrt{2(1-m^2)}}}+\frac{1}{2}\erf{\frac{\knew-Z_0}{\sqrt{2(1-m^2)}}}\right\}\,\\
  &+o\left(\frac{1-m}{2}\log{\left(\frac{1-m}{2}\right)}\right) \, ,\nonumber
\end{align}
in which $m$ is evaluated with the fixed-point equation
\begin{equation}
\label{eq: local max planting}
   -\log{\left(\frac{1-m}{2}\right)}+o\left(\log{\left[\frac{1-m}{2}\right]}\right)=
   \frac{4\alpha m}{{\mathcal{N}}_{\kold}}
   \int \mathcal{D}B\frac{\one\{\vert B\vert \le \kold\}}{\erf{\frac{\knew+B}{\sqrt{2(1-m^2)}}}+\erf{\frac{\knew-B}{\sqrt{2(1-m^2)}}}}   \left\{\frac{(\knew+B)e^{\frac{-(\knew+B)^2}{2(1-m^2)}}}{\sqrt{2\pi}(1-m^2)^{3/2}}+\frac{(\knew-B)e^{\frac{-(\knew-B)^2}{2(1-m^2)}}}{\sqrt{2\pi}(1-m^2)^{3/2}}\right\} \, .
\end{equation}

Using the rescaling from the previous section we can finally write for these two functions in the leading order in $\alpha \to 0$ :
\begin{align}
\label{eq: s planted}
     s[\tkold,\tknew]&= \frac{ \alpha\tilde{r}}{4}+\frac{\alpha }{{\mathcal{N}}_{\tkold}}\int\! \mathcal{D}B\, \one\{\vert B\vert \le \kold\}\,\log\left\{\frac{1}{2}\erf{\frac{\tknew+B}{\sqrt{2\tilde{r}}}}+\frac{1}{2}\erf{\frac{\tknew-B}{\sqrt{2\tilde{r}}}}\right\}+o({\alpha})\, ,\\
\label{eq: sigma planted}
     \Sigma[\tkold]&=\log(2)+\frac{\alpha}{2}\log\left( \frac{\alpha}{\log\alpha}\right)+\alpha\log\left({\tkold}\right)=\Sigma_o +\alpha\log\left(\sqrt{\frac{2}{\pi}}{\tkold}\right)\, ,
\end{align}
where we recall that $\tilde{r}$ is evaluated with
\begin{align}
      1&=\frac{4}{{\mathcal{N}}_{\tkold}}\int \mathcal{D}B\frac{\one\{\vert B\vert \le \kold\}}{\erf{\frac{\tknew+B}{\sqrt{2\tilde{r}}}}+\erf{\frac{\tknew-B}{\sqrt{2\tilde{r}}}}}   \left\{\frac{(\tknew+B)e^{\frac{-(\tknew+B)^2}{2\tilde{r}}}}{\sqrt{2\pi}\tilde{r}^{3/2}}+\frac{(\tknew-B)e^{\frac{-(\tknew-B)^2}{2\tilde{r}}}}{\sqrt{2\pi}\tilde{r}^{3/2}}\right\}+o({\alpha})\, 
\end{align}
and
\begin{align}
1-{m}^2= -\alpha  \tilde{r} /\log(\alpha)\, ,\;
\kold={\tkold} \sqrt{-\alpha /\log(\alpha)}\, ,\; 
\knew={\tknew} \sqrt{-\alpha /\log(\alpha)}\, , \;\Sigma_o=\log(2)+\alpha\log\left( \frac{-\alpha}{\log\alpha}\right)\, .
\end{align}

In Fig.~\ref{fig: local entropy and complexity} the right-hand side displays several curves of complexity $\Sigma[\tkold]$ as a function of the local entropy $s[\tkold,\tknew]$ for fixed values of $\tknew$. Three regimes can be outlined for $\knew < \kentro$. First, for $s[\tkold,\tknew]\approx 0$, we have locally convex curves (and $\tkold-\tknew=o(1)$). This result appears quite surprising as usually these $\Sigma(s)$ curves are fully concave \cite{braunstein2003polynomial,zdeborova2007phase}.
Then, the curve becomes concave while having $s[\tkold,\tknew]=\mathcal{O}(\alpha)$ and  $\tkold-\tknew=\mathcal{O}(1)$. In this regime, the complexity continues to scale as $\Sigma[\tkold]-\Sigma_o=\mathcal{O}(\alpha)$ and the local entropy is upper bounded by $s[0,\tknew]$. Finally, if we set $\tkold\ll \tknew$ (i.e. $\kold=o\left(\sqrt{-\alpha/\log(\alpha)}\right)\,$) the complexity jumps from $\Sigma[\tkold]\approx\Sigma_o$ to $\Sigma[\tkold]=0$. In this case, the entropy remains fixed (in first order) at $s[\tkold,\tknew]=s[0,\tknew]$. We sketched these three regimes for the complexity versus entropy curves in Fig.~\ref{fig:Complexity entropy sketch}.
For $\tknew > \tkentro$ only the first regime exists since for small enough $\kold$ the local maximum of the potential disappears. 

\begin{figure}[h!]
\hspace{-2cm}
    \centering
    \includegraphics[width=0.6\textwidth]{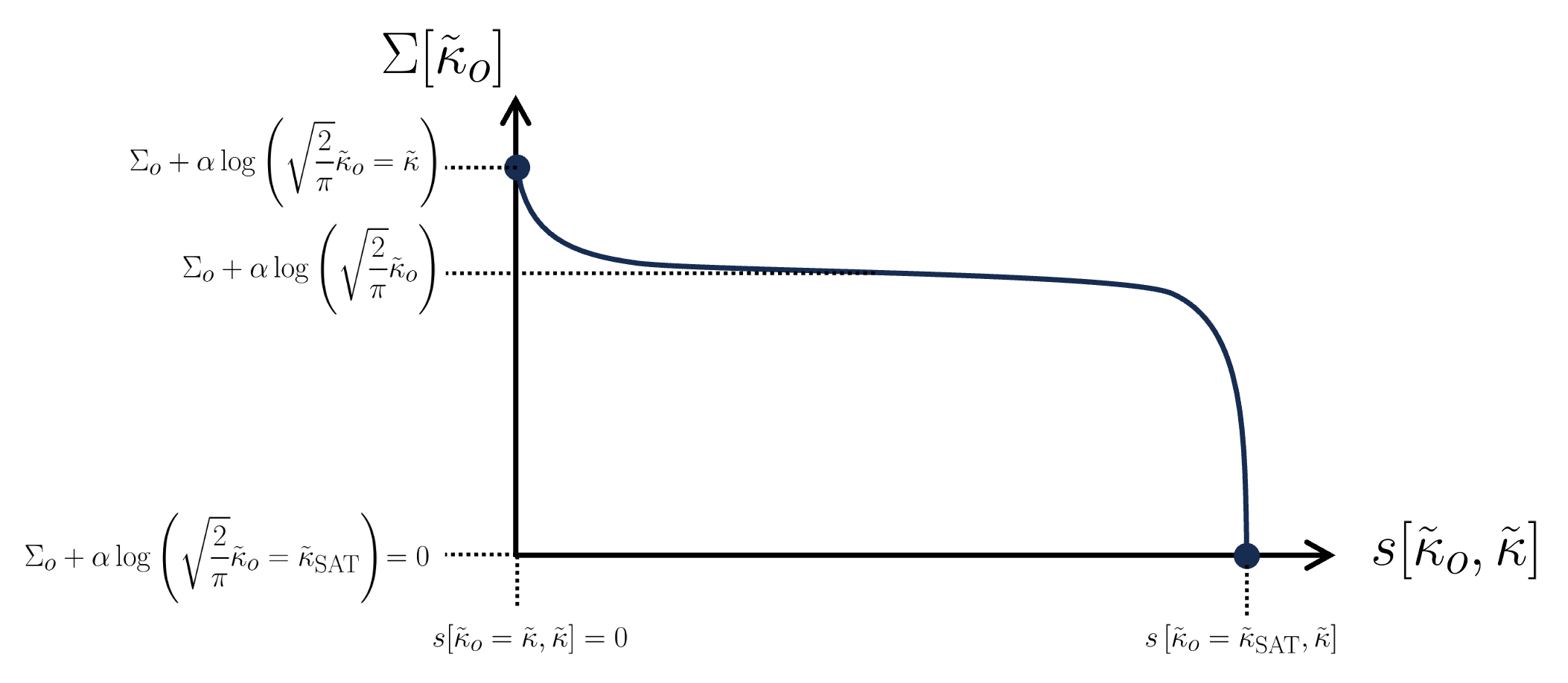}
    \caption{Sketch of the behavior of the complexity as a function of the local entropy. First, close to $s[\tkold,\tknew]=0$, we observe a regime in which the curve is convex. Then, as $\tkold$ is lowered (keeping $\tkold-\tknew=\mathcal{O}(1)$) the curve becomes concave. The complexity remains such that $\Sigma[\tkold]-\Sigma_o=\mathcal{O}(\alpha)$ while the entropy increases up until the upper bound $s[0,\tknew]$. Finally, as we set $\tkold\ll\tknew$, the entropy remains constant at $s[0,\tknew]$ while the complexity drops to zero. }
    \label{fig:Complexity entropy sketch}
\end{figure}

\section{Analysis of the clustered structure through the replica method}
\label{sec:1RSB}
\subsection{The 1-RSB free energy}
In this section, we show how the clustered structures we obtained with the planting approach can also be observed via the ordinary 1-RSB computation \cite{monasson1995structural}. For this we will consider the set of solutions $S(\bG, \knew)$ of the unbiased symmetric binary perceptron. In particular, we will consider its cardinality
\begin{align}
Z(\knew) := \big| S(\bG, \knew)\big|
\end{align}
and its total entropy function as the logarithm of $Z$ averaged over the disorder $\bG$
\begin{equation}
\phi_{N} := \frac{1}{N} \E_{\bG}\left[  \log Z( \knew ) \right]\, .
\end{equation}
So as to perform the average over the disorder we will use the replica trick \cite{monasson1995structural}. This trick takes the form of
\begin{align}
    \E_{\bG}\left[  \log Z( \knew ) \right]\underset{n\rightarrow 0}{=}\frac{ \E_{\bG}\left[ Z^n( \knew )\right]-1}{n}
\end{align}
where each of the $n$ introduced copies of the system is called a replica. This technique enables to shift from a computation where interactions are random and the replica decoupled to a computation where the replica interacts with deterministic couplings. With this approach, the rest of the computation mainly consists in evaluating the quantity $\E_{\bG}\left[ Z^n( \knew )\right]$ at a fixed-point of the overlap matrix $Q\in {\rm I\!R^{n\times n}}$, where
\begin{align}
    Q^{a,b}=\E_{\bG}\left[\by^a\cdot \by^{b}  \,\,\Big\vert\,\, \by^a,\by^b\in S(\bG, \knew) \right]\quad\text{and} \quad a,b\in[\![1,n]\!]\, .
\end{align}
Moreover, as the constraints on the overlaps are introduced in the following fashion
\begin{align}
    \delta\left(\by^a\cdot \by^{b}-Q^{a,b}\right)=\int d\hat{Q}^{a,b} e^{i \hat{Q}^{a,b} \left(\by^a\cdot \by^{b}-Q^{a,b}\right) }\underset{N\rightarrow +\infty}{=} \int d\hat{Q}^{a,b} e^{ \hat{Q}^{a,b} \left(\by^a\cdot \by^{b}-Q^{a,b}\right) }, 
\end{align}
we will have also to evaluated $\E_{\bG}\left[ Z^n( \knew )\right]$ at a fixed-point of the matrix  $\hat{Q}$. In more detail, the computation consists of evaluating
\begin{align}
\label{eq: general free energy}
    \E_{\bG}\left[ Z^n( \knew )\right]=\prod_{a=1}^n\left(\sum_{\by^a\in \{-1,1\}^N}\int\prod_{a<b}dQ^{a,b} d\hat Q^{a,b} e^{ \hat{Q}^{a,b} \left(\by^a\cdot \by^{b}-Q^{a,b}\right) }  \int \prod_{a=1}^n \prod_{j=1}^M \left[d v^{a,j} \Theta\big(\kappa-v^{a,j}\big)\right] e^{-\frac{1}{2}\sum_{a,b}v^{a,j} \Sigma^{a,b} v^{b,j}}\right) .
\end{align}
with
\begin{align}
    v^{a,j}=\bg_j\cdot \by^a\quad\text{and}\quad \Sigma^{a,b}=\E_{\bG}\left[(\bg_j\cdot \by^a)(\bg_j\cdot \by^b)\right]=\E_{\bG}\left[\by^a\cdot \by^b\right]=Q^{a,b}
\end{align}

The computation of $\E_{\bG}\left[ Z^n( \knew )\right]$ with the 1-step replica symmetric (1-RSB) {\it ansatz} implies the following form for the matrices $Q$ and $\hat Q$
\begin{align}
Q^{a,b}=\left\{
    \begin{array}{ll}
       1 & \text{if}\; a= b \\
       q_1 & \lfloor\frac{a}{x}\rfloor= \lfloor\frac{b}{x}\rfloor \\
       q_0 & \text{otherwise}\\
    \end{array}
\right.\! 
\quad\text{and}\quad
\hat Q^{a,b}=\left\{
    \begin{array}{ll}
       \hat{Q} & \text{if}\; a= b \\
       \hat q_1 & \lfloor\frac{a}{x}\rfloor= \lfloor\frac{b}{x}\rfloor \\
       \hat q_0 & \text{otherwise}\\
    \end{array}
\right.\! .
\end{align}
With this ansatz Eq.~(\ref{eq: general free energy}) boils down to
\begin{align}
\phi^{\rm 1-RSB}=\lim_{n \to 0}\, \lim_{N \to \infty} \frac{ \E_{\bG}\left[ Z^n( \knew )\right]-1}{nN}=&-\frac{x\hat{q}_1}{2}+\frac{x(1-x)q_1\hat{q}_1}{2}+\frac{x^2 q_0\hat{q}_0}{2}+\alpha\int \mathcal{D}t \, \log\left\{\int \mathcal{D}z \, e^{x\phi^{\knew}_{\rm out}[\sqrt{q_1-q_0}z+\sqrt{q_0}t,1-q_1]}\right\}\nonumber\\
&+\int \mathcal{D}t \, \log\left\{\int \mathcal{D}z \, e^{x\phi_{\rm in}[\sqrt{\hat{q}_1-\hat{q}_0}z+\sqrt{\hat{q}_0}t]}\right\}
\end{align}
with
\begin{align}
\phi^{\knew}_{\rm out}[\omega,V]&=\log\left[\int_{\frac{-\omega-\knew}{\sqrt{V}}}^{\frac{-\omega+\knew}{\sqrt{V}}} \mathcal{D}u\right]=\log\left[\frac{1}{2}{\rm erf}\left(\frac{\knew-\omega}{\sqrt{2V}}\right)+\frac{1}{2}{\rm erf}\left(\frac{\knew+\omega}{\sqrt{2V}}\right)\right]\, ,\\
\phi_{\rm in}[B]&=\log\left[\sum_{x=\pm 1}e^{Bx}\right]=\log\left[2\cosh{(B)}\right]\, .
\end{align}
For more details on the calculation steps to derive $\phi^{\rm 1-RSB}$ we redirect the interested readers to the first appendix of \cite{aubin2019storage}. Before moving on with the analysis of the 1-RSB potential, a first simplification consists in taking into account a symmetry in the in/out channels: $\phi^{\knew}_{\rm out}[\omega,V]=\phi^{\knew}_{\rm out}[-\omega,V]$ and $\phi_{\rm in}[B]=\phi_{\rm in}[-B]$. Indeed, this symmetry implies that optimizing the potential yields the solution $q_0=\hat{q}_0=0$. Thus, in the following, we will always take this solution. 
Then, the remaining equations we have to verify for the fixed point are
\begin{align}
    \label{eq: 1-RSB saddle 1}
    \hat{q}_1&=f(q_1)=-\frac{2\alpha}{(1-x)}\times \dfrac{\int \mathcal{D}z \,\partial_{q_1}\phi^{out}[\sqrt{q_1}z,1-q_1] e^{x\phi^{out}[\sqrt{q_1}z,1-q_1]}}{\int \mathcal{D}z \, e^{x\phi^{out}[\sqrt{q_1}z,1-q_1]}}\, ,\\
    \label{eq: 1-RSB saddle 2}
    q_1&=g(\hat q_1)=\frac{2}{1-x}\left[\frac{1}{2}-\dfrac{\int \mathcal{D}z \,\partial_{\hat{q}_1}\phi^{in}[\sqrt{\hat{q}_1}z] e^{x\phi^{in}[\sqrt{\hat{q}_1}z]}}{\int \mathcal{D}z \, e^{x\phi^{in}[\sqrt{\hat{q}_1}z]}}\right]\, .
\end{align}

With these definitions, the entropy and complexity of the clusters can be determined at the fixed point as
\begin{align}
\label{eq: 1-RSB defs}
    s=\partial_x \phi^{\rm 1-RSB}\, , \quad  \Sigma=\phi^{\rm 1-RSB}-xs\quad \text{and}\quad \frac{\partial \Sigma}{\partial s}=-x\, .
\end{align}

\subsection{The 1RSB solution at finite $\alpha$}

When it comes to solving the 1RSB equations, we focus in this subsection on $\alpha=0.5$ as a representative value not close to zero, the corresponding satisfiability threshold is $\kappa_{\sSAT}(\alpha=0.5)=0.319$. We obtained four branches of solutions when solving the fixed-point equations (\ref{eq: 1-RSB saddle 1}, \ref{eq: 1-RSB saddle 2}) with respect to $q_1$ and $\hat{q}_1$ for the 1-RSB potential (and browsing through values for the Parisi parameter $x$). Two of these solutions are unstable under the iteration scheme
\begin{align} \label{eq:iterations}
    \hat{q}_1^{t+1}=f(\hat{q}_1^t)\, ,\quad  {q}_1^{t+1}=g(q_1^t)\, ,
\end{align}
while the remaining two are stable. When browsing different values of $x$, we also observe a threshold value for $\knew$ for which the overall behavior of these fixed points changes. We will call this value $\kappa_{\rm break}(\alpha=0.5)\approx 0.455$. 
In Fig.~\ref{fig: 1-RSB comp} (left panel) we plot the complexity $\Sigma$ as a function of their entropy $s$ for the four branches. When tuning $x$ each solution describes a trajectory that we highlighted with either a dashed (unstable fixed point) or a full line (stable fixed point).
One key question arising from these results is how we should select the fixed-point branch that corresponds to the actual clusters of solutions in the problem. First, we clearly need to restrict to non-negative $\Sigma$ and non-negative $s$. Moreover, we know that the correct equilibrium state is given by the solution where the total entropy 
\begin{equation}
   s_{\rm total} = \Sigma+s ~\big|~ \Sigma \ge 0\, ,~ s\ge 0 
\end{equation}
is maximized. For the present model, this happens for $s=0$ when the (negative) slope of the $\Sigma(s)$ curve is infinite. This can be seen by realizing that the slope of the curve $\Sigma(s)$ is much smaller than $-1$. We recall that this slope is equal to $-x$, where $x$ is the Parisi parameter, as explained in Eq.~(\ref{eq: 1-RSB defs}). We highlighted this equilibrium point with a colored dot in the left panel of Fig.~\ref{fig: 1-RSB comp}. In particular, this point $\Sigma(0)$ corresponds to the equilibrium frozen 1RSB solution of the SBP problem with a value corresponding to one computed in \cite{aubin2019storage}. We note that this criterion for equilibrium is rather unusual among other models where the 1RSB solution was evaluated. Usually, either both $\Sigma>0$ and $s>0$ at the point where the negative slope $x =1$ corresponding to the so-called dynamical-1RSB phase, or the maximum is achieved when $\Sigma=0$ at a (negative) slope strictly between $0< x <1$ corresponding to the so-called static-1RSB phase. Here, we observe the equilibrium being achieved for $x\rightarrow +\infty$ corresponding to frozen-1RSB at equilibrium.  
Finally, we observe that for $\knew>\kappa_{\rm break}  \sim  0.455$ (still considering $\alpha=0.5$) the curve $\Sigma(s)$ for positive values of both $s$ and $\Sigma$ breaks into two branches.
Consequently, there is a finite range of values for the entropy $s$ where we do not obtain any fixed point. The meaning of such a gap is unclear, but it appears in other problems and their 1-RSB solution \cite{zdeborova2008constraint}.

\begin{figure}[!ht]
    \centering
    \includegraphics[width=0.49\textwidth]{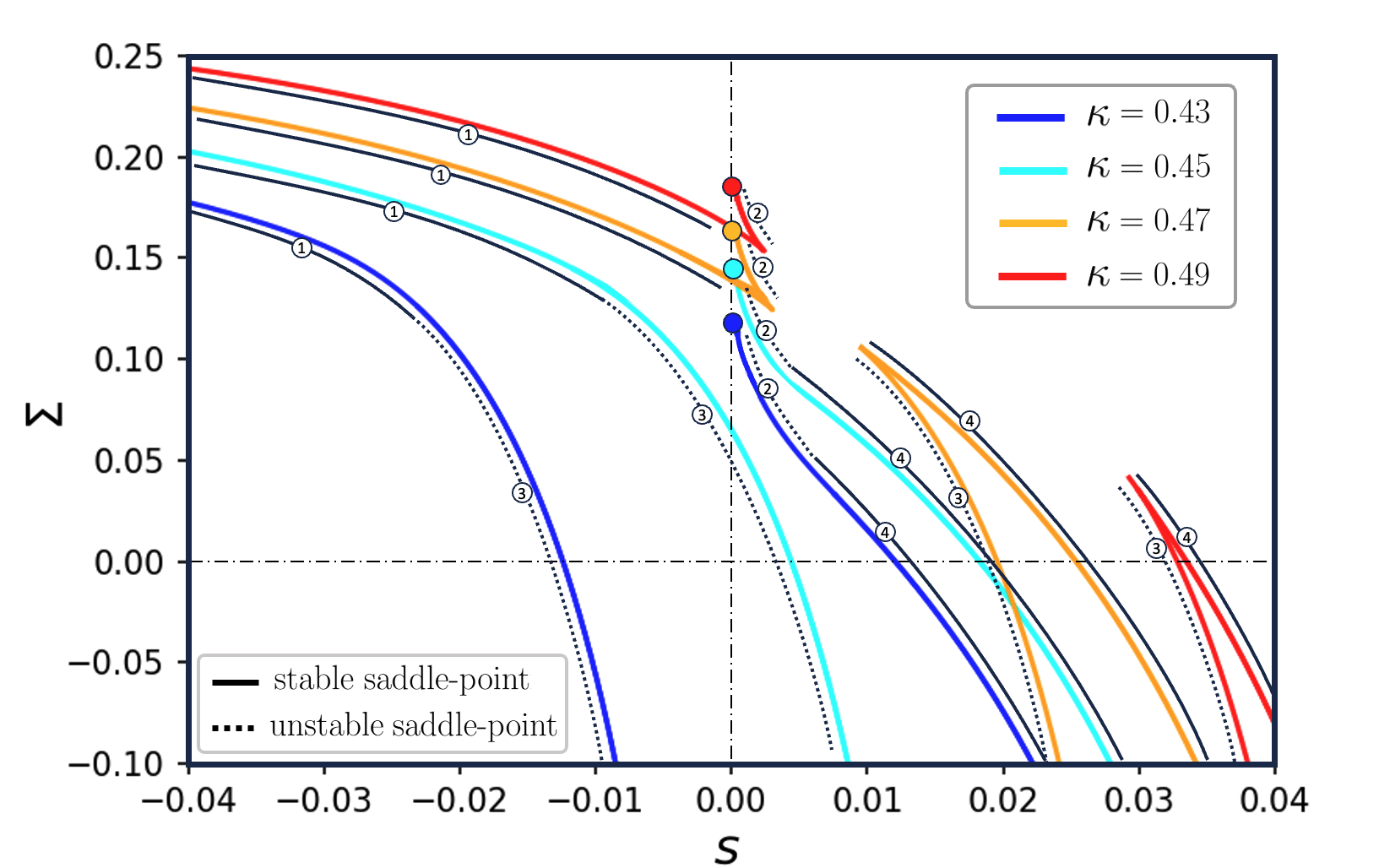}
    \includegraphics[width=0.49\textwidth]{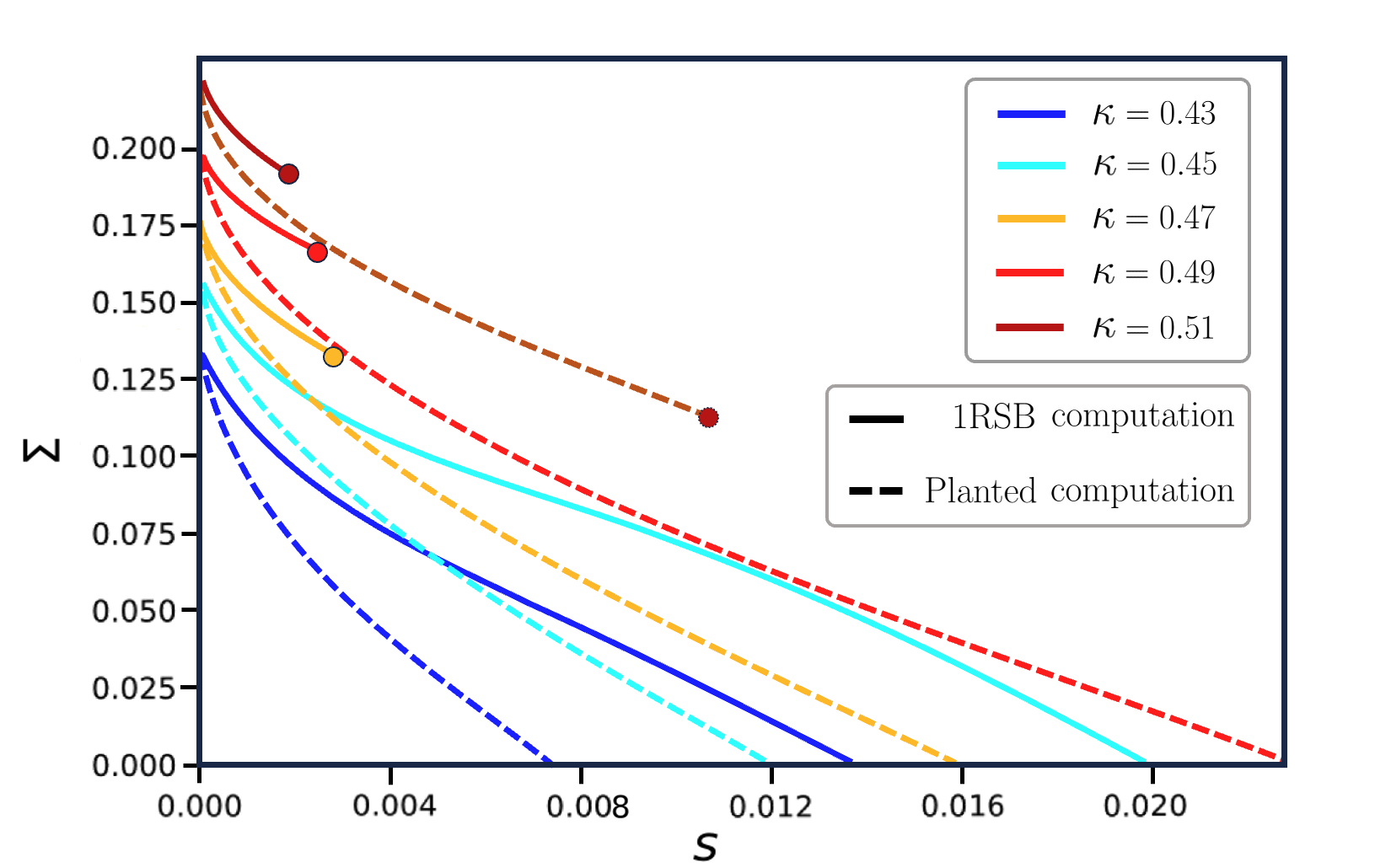}
    \caption{We plot in these two panels the complexity $\Sigma$ as a function of the entropy $s$, in both cases $\alpha=0.5$. On the left panel, we plot all the branches obtained when solving the saddle-point equations with respect to $q_1$ and $\hat{q}_1$ for the 1-RSB potential (and browsing through values for the Parisi parameter $x$). As a guide-to-the-eye we emphasize each of the four branches of solutions with either a full or dashed line. The full lines correspond to stable fixed points regarding the iteration scheme of Eq.~(\ref{eq:iterations}), while dashed ones correspond to unstable fixed points. We highlighted with colored dots the fact that certain branches stop at $s=0$ at a value of $\Sigma$ corresponding to the equilibrium solutions. To reach this point, the Parisi parameter $x$ has to be set to infinity.
    On the right panel, we compare the results obtained by the branches yielding this equilibrium complexity with the ones obtained with the previous planting method. We added colored dots when the fixed point (either of the planted or the 1-RSB saddle-point equations)
    with maximum entropy $s$ { was}  obtained for $\Sigma=0$.}
    \label{fig: 1-RSB comp}
\end{figure}

In Fig.~\ref{fig: 1-RSB comp} right panel, we plot the complexity $\Sigma$ as a function of the entropy $s$ selecting the branch that is an analytic continuation of the equilibrium $\Sigma(0)$ point and compare it to the one obtained via the planting approach. For this comparison, we need the local entropy in the planted model at finite values of $\alpha$ that is derived in appendix \ref{app: A}. 
We note that the two complexities exactly agree at $s=0$ as is expected because at $s=0$ both the complexities correspond to the total number of solutions at that $\kappa$. For $s>0$, the
two complexities have a similar shape, being clearly convex for small values of $s$. 
We note again that the overall values of $\Sigma$ are larger than the values of $s$ meaning that the slope actually takes a rather large value in the whole range of those curves. We recall that in the context of the 1-RSB computation this slope is equal to $-x$, where $x$ is the Parisi parameter. Taking its value much larger than one is not common in other models for which 1RSB was studied. This is likely the reason why these extensive size clusters were not described earlier in the literature for the binary perceptron.  
We further see that the 1RSB complexity, when it exists, is strictly larger than the one obtained via planting as again expected since via planting we obtain only some of the clusters of solution whereas the 1RSB computation should be able to count all of them. 
Then, when $\kappa > \kappa_{\rm break}$, the planted model predicts the existence of clusters with an internal entropy that lies inside the fixed-point gap of the 1-RSB approach. This indicates that the 1RSB solution does not fully describe the space of solutions in this case. This may have many causes. For example,  we may have missed a branch of fixed points in our analysis of the 1-RSB potential. Or, this region may involve a replica \emph{ansatz} with further symmetry breaking. Or perhaps these rare clusters simply cannot be obtained with a replica computation.
Finally, when $\knew>\kenerg(\alpha=0.5)\approx 0.499$,
the curves $\Sigma(s)$  obtained with planting stop at some positive values of $\Sigma$ and $s$ and thus look again qualitatively similar to the portion of the curve $\Sigma(s)$ that is obtained from 1RSB by analytically continuing from the equilibrium $\Sigma(s=0)$ point. 

Overall, the 1RSB approach evaluated at sufficiently larger values of the Parisi parameter $x$ identified clusters of extensive size in parts of the solution corresponding to a convex curve that is unstable under the iterations of the 1RSB fixed point equations. These curves are partly compatible with the complexity obtained from planting. Yet there are still regions of $\kappa, s$ for which we obtain extensive clusters of solution from the planting procedure but not from the 1RSB. The reason behind this paradox is left for future work. 

For small $\alpha$, the situation becomes actually clearer. In the next section, we will discuss this case.

\subsection{The $\alpha \to 0$ and $x\rightarrow+\infty$ limit}
We now focus, as in the first part of the paper, on the regime of small $\alpha$. 
Using our results from the planting computation, and anticipating a similarity of behaviour in the 1RSB, we can deduce the behavior of the Parisi parameter $x$ in the low $\alpha$ limit. Indeed, alike the 1-RSB computation, we saw that the planting approach probes clustered solutions. It also allows for computing their complexity $\Sigma$ and local entropy $s$, see Eq.~\eqref{eq: s planted} and \eqref{eq: sigma planted}. As mentioned above, in the context of a 1-RSB computation we have $\partial\Sigma/\partial s=-x$. Thus, if we plug-in the entropy and complexity from Eq.~\eqref{eq: s planted}, \eqref{eq: sigma planted} we can compute $\partial\Sigma/\partial s$ and estimate the Parisi parameter. 

By doing so we obtain two regimes for which the slope $\partial\Sigma/\partial s$ becomes infinite in the low $\alpha$ limit. First, when $\tkold\ll\tknew$ the entropy remains constant at first order in $\alpha$ while the complexity roughly jumps from $\Sigma_o$ to zero, see the left panel in Fig.~\ref{fig: local entropy and complexity}. It indicates that to recover these states with the 1-RSB computation we should set $x\gg1$.
The second regime for which we observe an infinite slope corresponds to $\tkold\approx\tknew$. Indeed, we have $\vert\partial \Sigma/ \partial s \vert\sim (\tknew-\tkold)^{-1}$ close to $\tkold=\tknew$. Consequently, if we want to probe the clusters with almost zero local entropy we should also set $x\gg1$ to obtain this regime. 

A last piece of information given by the planted model is that these clusters (in the two regimes mentioned above) correspond to a limit where $q=m^2\approx 1$. Thus, if we impose the same condition in the 1-RSB fixed-point equations,  we have that setting $q_1\approx1$ in Eq.~(\ref{eq: 1-RSB saddle 2}) implies $\hat q_1\gg 1$. Therefore, in order to find these clusters we will not only set $x\gg 1$ but we will also take $q_1\approx 1$ and  $\hat q_1\gg 1$.

The first regime we mentioned will be referred as the {\it maximum entropy} regime, while the second one will be referred as the {\it minimum entropy} regime. First, we see that the entropic contribution can be simplified identically in both regimes. Indeed, when setting $\hat q_1\gg1$ we obtain
\begin{align}
    \phi_{\rm in}\left[\sqrt{\hat{q}_1}z\right]&{\approx}\sqrt{\hat{q}_1}\vert z\vert+\log\left[1+e^{-2\sqrt{\hat{q}_1}\vert z\vert}\right]\\
    &{\approx}\sqrt{\hat{q}_1}\vert z\vert+e^{-2\sqrt{\hat{q}_1}\vert z\vert}\nonumber
\end{align}
which then yields
\begin{align}
    \log\left\{\int \mathcal{D}z \, e^{x\phi_{\rm in}[\sqrt{\hat{q}_1}z]}\right\}&\approx \log\left\{2\int_0^{+\infty} \mathcal{D}z \,e^{x\sqrt{\hat{q}_1}z+xe^{-2\sqrt{\hat{q}_1}z}} \right\}\\
    &\approx \log\left\{2\int_0^{+\infty} \mathcal{D}z \,e^{x\sqrt{\hat{q}_1}z}\left(1+xe^{-2\sqrt{\hat{q}_1}z}\right) \right\}\nonumber\\
    &\approx  \log\left\{2 \Bigg(e^{\frac{x^2 \hat{q}_1}{2}}+xe^{\frac{(x-2)^2 \hat{q}_1}{2}}\Bigg) \right\}\nonumber\\
    &\approx \log 2+\frac{x^2 \hat{q}_1}{2}+x e^{-2(x-1) \hat{q}_1}\, .\nonumber
\end{align}

As we will see in the following subsections, the simplification for the energetic term will be regime-dependent.

\subsubsection{Maximum entropy regime}
For the maximum entropy regime, we can again go back to the results from the planting model to help us make the correct approximation. We know, for example, that we should have $1-q_1\sim \kappa^2$ in the low $\alpha$ limit (as both quantities have the same scaling in $\alpha$). This implies that we should have
\begin{eqnarray}
\label{eq: order mag}
     \phi^{\knew}_{\rm out}[\sqrt{q_1}z,1-q_1]=\mathcal{O}(1) \quad \text{with}\quad z\in[-\knew,\knew].
\end{eqnarray}

Therefore, with $x\gg1$, we will approximate the energetic term with a saddle-point method. In other words, we will compute
\begin{align}
   \int \mathcal{D}z \, e^{x\phi^{\knew}_{\rm out}[\sqrt{q_1}z,1-q_1]}\approx e^{x\phi^{\knew}_{\rm out}[\sqrt{q_1}Z_0,1-q_1]} \int \mathcal{D}z \, e^{x\frac{q_1 (z-Z_0)^2}{2}\partial^2_\omega\phi^{\knew}_{\rm out}[\sqrt{q_1}Z_0,1-q_1]}
\end{align}
where $Z_0$ corresponds to the maxima of $\phi^{\knew}_{\rm out}[\sqrt{q_1}z,1-q_1]$. In particular, with this function we have $Z_0=0$.
By doing the saddle-point approximation in $Z_0=0$ we obtain
\begin{align}
    \alpha\log\left\{\int \mathcal{D}z \, e^{x\phi^{\knew}_{\rm out}[\sqrt{q_1}z,1-q_1]}\right\} \approx \alpha x \log\left[{\rm erf}\left(\frac{\knew}{1-q_1}\right)\right]-\frac{\alpha}{2}\log\left[1+\Delta x\right]
\end{align}
with
\begin{align}
    \Delta=\left\vert q_1 \partial_\omega^2 \phi^{\knew}_{\rm out}(0,1-q_1) \right\vert \, .
\end{align}
Finally, if we combine the simplification of both the entropic and energetic contributions the total entropy becomes
\begin{align}
\label{eq: 1-RSB free energy regime 1}
 \phi^{1-RSB}\approx&-\frac{x\hat{q}_1}{2}+\frac{x(1-x)q_1\hat{q}_1}{2}+\alpha x \log\left[{\rm erf}\Big(\frac{\knew}{\sqrt{2(1-q_1)}}\Big)\right]   \\
 &-\frac{\alpha}{2}\log(1+\Delta x)+\log(2)+\frac{x^2\hat{q}_1}{2}+xe^{-2(x-1)\hat{q}_1}\, \nonumber
\end{align}
and its fixed point equations are, at first order in $x$,
\begin{align}
\label{eq: 1-rsb saddle 1 ansatz 2}
    \frac{x(x-1)\hat{q}_1}{2}&=\frac{\alpha x \knew e^{\frac{-\knew^2}{2(1-q1 )}}}{\sqrt{2\pi}(1-q_1)^{3/2}  {\rm erf}\Big(\frac{\knew}{\sqrt{2(1-q_1)}}\Big)}\, ,\\
\label{eq: 1-rsb saddle 2 ansatz 2}
    q_1&=1-4e^{-2(x-1)\hat{q}_1}\, .
\end{align}

Now, we are able to draw a direct parallel with the planted system at low $\alpha$ and $\kold\ll \knew$. Indeed, if we use the correspondence $q_1\equiv m^2$ and $(x-1)\hat{q}_1\equiv \hat{m}$, Eqs.~(\ref{eq: 1-rsb saddle 1 ansatz 2},\ref{eq: 1-rsb saddle 2 ansatz 2}) are nothing but the fixed-point equations for the planted model (see Eqs.~\ref{eq: saddle point m approx 1}, \ref{eq: saddle point m approx 2}). This indicates a posteriori that the 1-RSB calculation enables us to recover the same clusters as in the planted model where we have set $\kold\ll \knew$.
To make the identification between the two approaches even more direct we can focus on the entropy and the complexity of this 1-RSB fixed-point. We obtain at first order in $x$ that
\begin{align}
    s&=\partial_x \phi^{1-RSB} =\frac{(1-q_1)x\hat{q}_1}{2}+\alpha\log\left[{\rm erf}\Big(\frac{\knew}{\sqrt{2(1-q_1)}}\Big)\right]\nonumber\\
    &=\left. \phi\left(r=\frac{1-m}{2}\right)\right\vert_{\alpha\ll 1,\,\kold\ll \knew,\, m\approx1}
\end{align}
and
\begin{align}
     \Sigma=\phi^{1-RSB}-xs=0\, .
\end{align}
Thus, at first order in $x$ these clustered states have exactly the same entropy and complexity as the ones from the planted system with $\kold\ll\knew$ (and $\alpha\ll 1$).

\subsubsection{Minimum entropy regime}
In this section, we want to probe clusters with a very small entropy. Now, keeping $\knew$ fixed, this means that we will have to set $q_1$ extremely close to one and eventually have $1-q_1\ll\kappa$ in order to go up to zero local entropy. This scaling between $q_1$ and $\knew$ is incompatible with the saddle-point approximation we performed for the maximum entropy regime, as Eq.~(\ref{eq: order mag}) is not verified anymore. In fact, in this case, we have to introduce an asymptotic expansion
\begin{align}
\label{eq: simplification phi_out}
     \phi^{\knew}_{\rm out}[\sqrt{q_1}z,1-q_1]&\approx\log\left\{\Theta\big(\knew/\sqrt{q_1}-\vert z\vert \big)-\sqrt{\frac{1-q_1}{2\pi}}\left[\frac{e^{\frac{-(\knew+\sqrt{q_1}z)^2}{2(1-q_1)}}}{\knew+\sqrt{q_1}z}+\frac{e^{\frac{-(\knew-\sqrt{q_1}z)^2}{2(1-q_1)}}}{\knew-\sqrt{q_1}z}\right]\right\}
\end{align}
where we used the identity
\begin{align}
    {\rm erf}(x)\underset{x\rightarrow+\infty}{\approx}1-\frac{e^{-x^2}}{x\sqrt{\pi}}\, .
\end{align}

We then estimate the interval of value of $z$ ($z\in[-Z_0,Z_0]$) for which $e^{x\phi^{\knew}_{\rm out}[\sqrt{q_1}z,1-q_1]}$ remains finite. In other words, we compute the value of $z$ for which the function $e^{x\phi^{\knew}_{\rm out}[\sqrt{q_1}z,1-q_1]}$ is equal to an arbitrary value $1/C$,
\begin{align}
   & e^{x\phi^{\knew}_{\rm out}[\sqrt{q_1}z,1-q_1]}=\frac{1}{C}\\
    \implies & x\phi^{\knew}_{\rm out}[\sqrt{q_1}z,1-q_1]=-\log C \nonumber\\
    \implies & \sqrt{\frac{1-q_1}{2\pi}}\times\frac{e^{\frac{-(\knew-\sqrt{q_1}z)^2}{2(1-q_1)}}}{\knew-\sqrt{q_1}z}\underset{\frac{\knew}{\sqrt{q_1}}>z>0}{=}\frac{\log C}{x}\nonumber\\
    \implies & \frac{\knew-\sqrt{q_1}z}{\sqrt{2(1-q_1)}}\underset{\frac{\knew}{\sqrt{q_1}}>z>0}{=}\sqrt{\frac{W_0\left(\frac{2}{a^2}\right)}{2}}\quad \text{with}\quad a=\frac{2 \log C}{x\sqrt{\pi}}\nonumber\\
     \implies & \frac{\knew-\sqrt{q_1}z}{\sqrt{2(1-q_1)}}\underset{\underset{\frac{\knew}{\sqrt{q_1}}>z>0}{ x\rightarrow +\infty}}{=}\sqrt{\frac{\log{x^2}}{2}}\nonumber\\
     \implies & z \underset{\underset{\frac{\knew}{\sqrt{q_1}}>z>0}{ x\rightarrow +\infty}}{=}\frac{\knew-\sqrt{2(1-q_1)\log x}}{\sqrt{q_1}}\, \nonumber
\end{align}
where $W_0(.)$ is the Lambert function with branch index $k=0$.
This computation thus shows that $e^{x\phi^{\knew}_{\rm out}[\sqrt{q_1}z,1-q_1]}$ jumps from $1$ to any arbitrary fraction $1/C$ exactly at 
\begin{align}
    Z_0 &\underset{\underset{\frac{\knew}{\sqrt{q_1}}>z>0}{ x\rightarrow +\infty}}{=}\frac{\knew-\sqrt{2(1-q_1)\log x}}{\sqrt{q_1}}\, . 
\end{align}
In this limit, we thus have for the energetic contribution
\begin{align}
    \alpha \log\left\{\int \mathcal{D}z \, e^{x\phi^{\knew}_{\rm out}[\sqrt{q_1}z,1-q_1]}\right\}&={\alpha}\log\left\{\int_{-Z_0}^{Z_0}{D}z\right\}={\alpha}\log\left\{{\rm erf}\left(\frac{\knew-\sqrt{2(1-q_1)\log x}}{\sqrt{2 q_1}}\right)\right\}\, .
\end{align}
And finally, if we put together the simplified entropic and energetic contribution to the 1-RSB potential we obtain
\begin{align}
 \label{eq: 1-RSB free energy regime 2}
 \phi^{1-RSB} \approx&-\frac{x\hat{q}_1}{2}+\frac{x(1-x)q_1\hat{q}_1}{2}+{\alpha}\log\left\{{\rm erf}\left(\frac{\knew-\sqrt{2(1-q_1)\log x}}{\sqrt{2 q_1}}\right)\right\}   \\
 &+\log(2)+\frac{x^2\hat{q}_1}{2}+xe^{-2(x-1)\hat{q}_1}\, \nonumber
\end{align}
and the corresponding fixed point equations are (at first order in $x$)
\begin{align}
\label{eq: 1-rsb saddle 1 ansatz 3}
   \frac{x^2\hat{q}_1}{2}&=\alpha\frac{e^{-{\knew}^2/2}}{{\rm erf}\big(\frac{\knew}{\sqrt{2}}\big)}\sqrt{\frac{\log x}{\pi (1-q_1)}}\, ,\\
\label{eq: 1-rsb saddle 2 ansatz 4}
    q_1&=1-4e^{-2(x-1)\hat{q}_1}\, .
\end{align}

The combination of these two fixed-point equations implies $\knew\gg\sqrt{2(1-q_1)\log x}$.
Consequently, the term $x^2(1-q_1)\hat{q}_1$ can be neglected and the 1-RSB free energy boils down to
\begin{align}
     \phi^{1-RSB}\approx&\alpha\log\left\{{\rm erf}\left(\frac{\knew}{\sqrt{2 }}\right)\right\} +\log(2)\, .
\end{align}
We thus recover the case of the planted system at $\kold=\knew$ as we have
\begin{align}
        s&=\partial_x \phi^{1-RSB} =0\, ,\\
        \Sigma&=  \phi^{1-RSB}\approx\alpha\log\left\{{\rm erf}\left(\frac{\knew}{\sqrt{2 }}\right)\right\} +\log(2)\,.
\end{align}
In \cite{aubin2019storage} the authors showed in a more standard computation that these equilibrium configurations (verifying a frozen 1-RSB structure) can also be obtained by imposing $q_0=0$, $q_1=1$ and $x=1$.

 In Fig.~\ref{fig: planted vs 1-RSB} we display for several values of $\tknew=\knew\sqrt{-\log(\alpha)/\alpha}$ the complexity as a function of the entropy. The light-colored full lines correspond to the results obtained with the planted model. The dashed and dotted lines correspond respectively to the maximum and minimum entropy fixed-point branches of the 1-RSB free energy. To obtain them we solved the fixed-point equations in each regime for large but finite Parisi parameter $x$. 
  As shown by the previous computations each end of the curve sees a close match between the planting approach and one of the $x\rightarrow+\infty$ regimes.
  This leads us to conjecture that in the limit of small $\alpha$, the planted and 1-RSB $\Sigma(s)$ curves match exactly. In other words, the planting approach actually captures the dominant clusters for each size $s$. 

\begin{figure}[!th]
    \centering
    \includegraphics[width=0.5\textwidth]{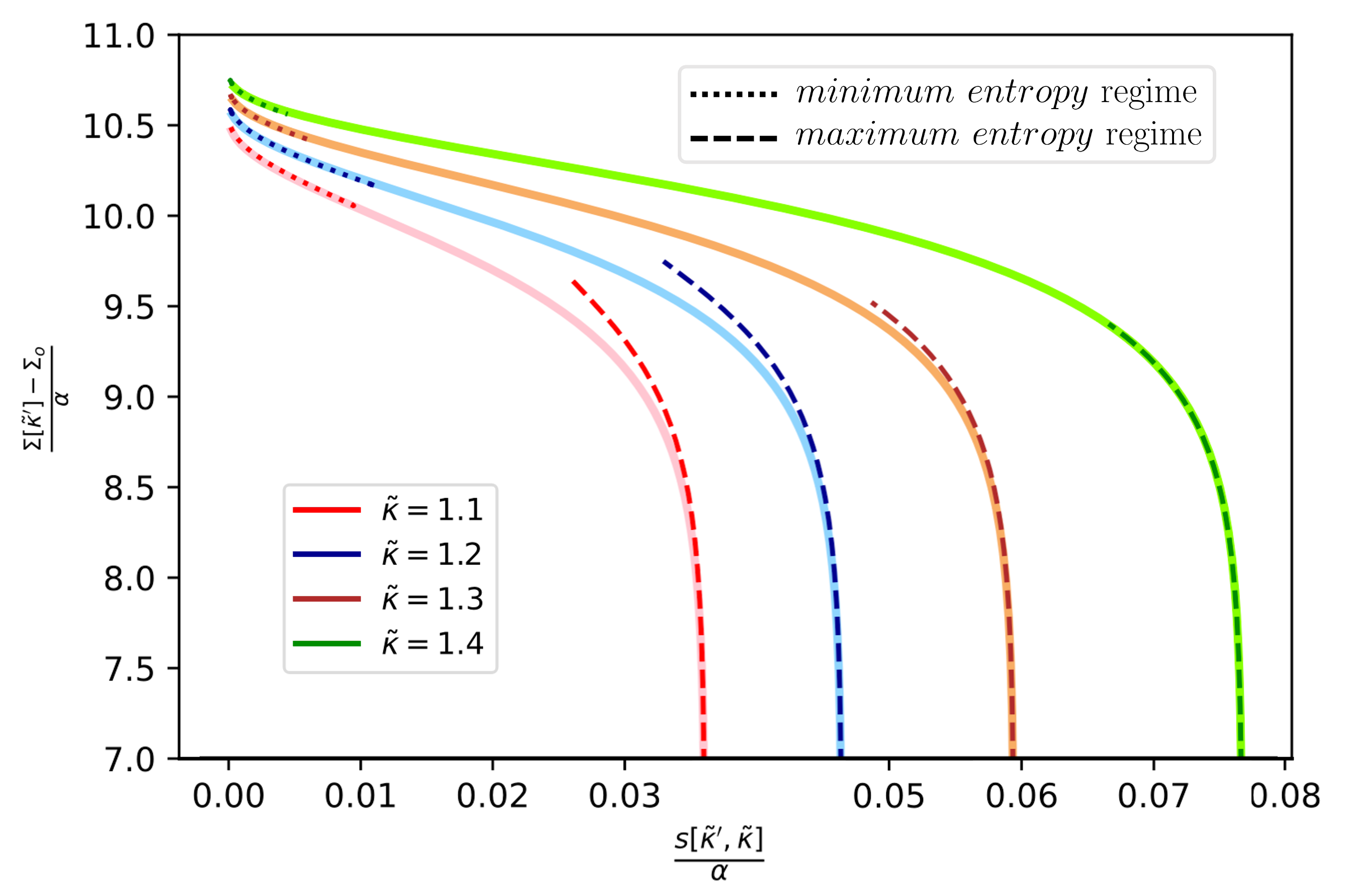}
    \caption{We plot the complexity $\Sigma$ as a function of the entropy $s$ for several values of $\tknew$. The dashed and dotted curves correspond to the maximum and minimum entropy regime and are obtained by optimizing the potentials Eqs.~(\ref{eq: 1-RSB free energy regime 1}) and (\ref{eq: 1-RSB free energy regime 2}) over $q_1$ and $\hat{q}_1$ with large but finite values of the Parisi parameter $x$. 
    In order to compare the 1-RSB computation with the planting approach, the full curves correspond to the complexity and entropy obtained with the planting computation.}
    \label{fig: planted vs 1-RSB}
\end{figure}

\vspace{-4mm}

\section{Conclusion and discussion}
\label{sec:conclusion}

We study the local entropy in the SBP problem around solutions planted at a smaller margin. Our results are rigorous in the limit of small $\alpha$, conditional on a condition of concentration of a certain entropy. We identify clusters of solutions of an extensive entropy as local maximizers of this local entropy. We identify two thresholds $\kenerg$ and $\kentro$ that we consider of particular interest.  $\kentro$ is the smallest $\knew$ at which the planted clusters and the corresponding maximum disappear, thus presumably melting into an extended structure that may be accessible to efficient algorithms. $\kenerg$ is a value above which there are solutions at all distances from the planted solutions and as such is an upper bound on the overlap gap property threshold.

We then investigated the 1RSB solution of the symmetric binary perceptron problem and showed how it allows us to identify extensive clusters of solutions without introducing concepts that are not present in the canonical 1RSB computation already. It suffices to consider large values of the Parisi parameter $x$ and both convex and concave parts of the $\Sigma(s)$ curve. We discuss how the equilibrium frozen-1RSB is recovered in the $x \to \infty$ limit. While this resolves some open questions about the 1RSB solution for binary perceptions, we conclude that the 1RSB calculation is incomplete at finite $\alpha$ as we did not find solutions corresponding to all the extensive clusters identified by the planting procedure.

We further showed that while, in general, the planting procedure we study does not describe all the rare clusters, in the limit of small $\alpha$ it seems that the $\Sigma(s)$ obtained via planting is exactly the same as the one obtained from the 1RSB. This leads us to conjecture that in the limit of small $\alpha$ the planting actually describes almost all clusters of a given size.

\acknowledgments
We acknowledge funding from the Swiss National Science Foundation grants OperaGOST (grant number 200390) and SMArtNet (grant number 212049).
We also thank David Gamarnik, Carlo Lucibello and Riccardo Zecchina for enlightening discussions on these problems.

\newpage


\newpage
\appendix

\section{The local entropy in the planted model}
\label{app: A}
\subsection{The local entropy for generic parameters $\alpha$, $\kold$ and $\knew$}

While in the main text we derive the local entropy rigorously at small $\alpha$, in this appendix we give the expressions for a generic value of $\alpha$. This we have not established rigorously and instead resorted to the replica method applied to the contiguous planted system. 
Following \cite{aubin2019storage} the replica symmetric (RS) free energy  for the planted binary perceptron is
\begin{align}
\label{eq: true local free energy}
\phi^{\kold,\knew}_{\rm planted}[q,\hat{q},m,\hat{m}]=\frac{1}{N}\E_{\bG,\bx_0} \log \left[Z\left(\bx_0 , \knew , r=\frac{1-m}{2}\right) \right] =&-\frac{1-q}{2}\hat{q}-m\hat{m}+\alpha\int \mathcal{D}z\,  dw P_{\kold}[w]\,\phi^{\knew}_{\rm out}[w,z,q,m]\\
&+\int\mathcal{D}z \,\left\{\phi_{\rm in}[z,\hat{q},\hat{m}]\right\} \, ,\nonumber
\end{align}
with
\begin{align}
\phi^{\knew}_{\rm out}[w,z,q,m]&=\log(\Phi^{\knew}_{\rm out}[\omega=mw+\sqrt{q-m^2}z,V=1-q])\\
&=\log\Bigg[\int_{\frac{-mw-\sqrt{q-m^2}z-\knew}{\sqrt{1-q}}}^{\frac{-mw-\sqrt{q-m^2}z+\knew}{\sqrt{1-q}}} \mathcal{D}u\Bigg]\nonumber\\
&=\log\Bigg\{\frac{1}{2}\erf{\frac{\knew-mw-\sqrt{q-m^2}z}{\sqrt{2(1-q)}}}\nonumber\\
&\qquad      +\frac{1}{2}\erf{\frac{\knew+mw+\sqrt{q-m^2}z}{\sqrt{2(1-q)}}}\Bigg\}\, ,\nonumber\\
\phi_{\rm in}[z,\hat{q},\hat{m}]
&=\log(\Phi_{\rm in}[B=\hat{m}x_o+\sqrt{\hat{q}}z])\\
&=\log\Bigg[\sum_{x=\pm 1}e^{(\hat{m}x_o+\sqrt{\hat{q}}z)x}\Bigg]\nonumber\\
&=\log\left[2\cosh(\hat{m}x_o+\sqrt{\hat{q}}z)\right]\nonumber
\end{align}
and where $\mathcal{D}z$ (as well as $\mathcal{D}x$ later) represents an integration with a scalar normal-distributed variable. Moreover, in the following we will take $P_{\kold}[w] = e^{-w^2/2}\one\{\vert w \vert \le \kold\}/\mathcal{N}_{\kold}$ where $\mathcal{N}_{\kold}$ is the prefactor normalizing the distribution. This distribution is given by the typical clusters dominating the probability measure $P^{\bm{g}}_{\kold}[\bm{x}]$. The variable $x_o$ corresponds to the planted configuration and can take the values $\pm 1$ indifferently. 

This free energy has saddle points for a set of parameters $\{q,\hat{q},m,\hat{m}\}$ verifying
 \begin{eqnarray}
 \label{eq: SE m}
 \hat{m}&=&\alpha\int \mathcal{D}z\,  dw P_{\kold}[w] \partial_m \phi^{\knew}_{\rm out}[w,z,q,m]\, ,\\
  \label{eq: SE q}
 \hat{q}&=&-2\alpha\int \mathcal{D}z\,  dw P_{\kold}[w] \partial_q \phi^{\knew}_{\rm out}[w,z,q,m]\, ,\\
 m&=&\int \mathcal{D}z\,x_o\, {\rm tanh}(\hat{m}x_o+\sqrt{\hat{q}}z)=\int \mathcal{D}z\, {\rm tanh}(\hat{m}+\sqrt{\hat{q}}z)       \, ,\\
 \label{eq: SE q_hat}
 q&=&1-\frac{1}{\sqrt{\hat{q}}}\int \mathcal{D}z\, z\, {\rm tanh}(\hat{m}x_o+\sqrt{\hat{q}}z)=\int \mathcal{D}z\, {\rm tanh}^2(\hat{m}+\sqrt{\hat{q}}z)\, .
 \end{eqnarray}

In Fig.~\ref{fig: Complexity vs entropy finite alpha} we plot the complexity, see Eq.~(\ref{eq: complexity}), as a function of the local entropy of the planted model. As explained in Sec.~\ref{subsec: Complexity vs entropy} the entropy is obtained by evaluating $\phi^{\kold,\knew}_{\rm planted}$ with its only non-trivial saddle-point over $\{q,m,\hat{q},\hat{m}\}$. The complexity is the exponential number of possible typical solutions for the binary perceptron at $\alpha$ and $\kold$, see Eq.~(\ref{eq: complexity}), which corresponds to the number of possibilities of planting the configuration. The results we display in this figure are obtained for $\alpha=10^{-2}$.

\begin{figure}
    \centering
    \includegraphics[width=0.5\textwidth]{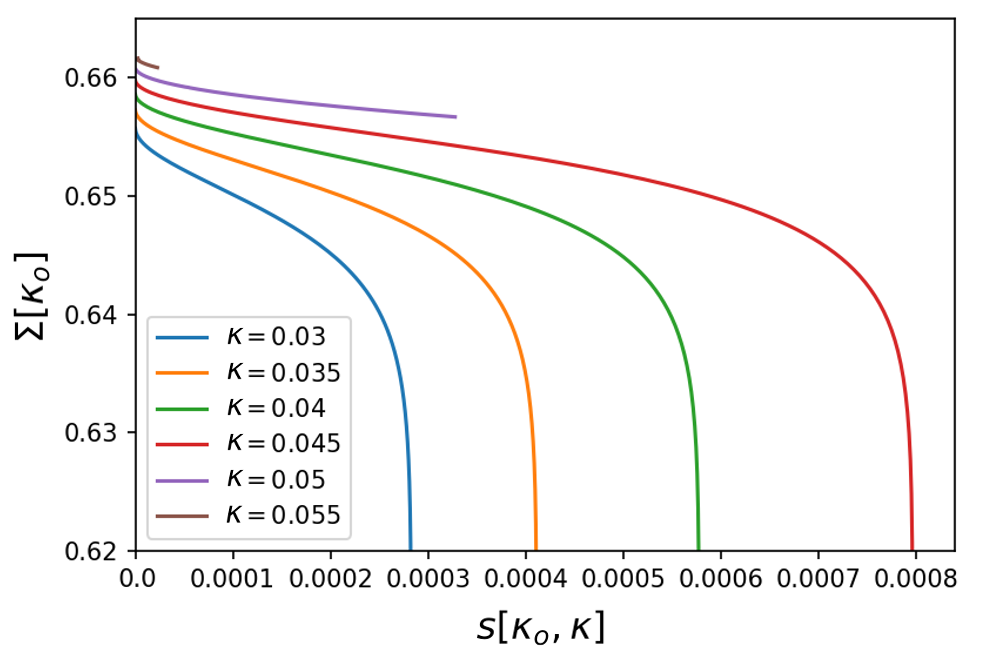}
    \caption{We plot the complexity $\Sigma[\kold]$ as a function of the local entropy $s[\kold,\knew]$ for $\alpha=10^{-2}$ and $\knew\in[0.03,0.06]$ using the planting approach. The function $s[\kold,\knew]$ is obtained by computing the non-trivial optimizer $\{q,m,\hat{q},\hat{m}\}$ for the local free entropy $\phi^{\kold,\knew}_{\rm planted}$ and evaluating its entropy, see Eqs.~(\ref{eq: true local free energy}) to (\ref{eq: SE q_hat}). The complexity corresponds to the exponential number of typical solutions for the binary perceptron at $\alpha$ and $\kold$, see Eq.~(\ref{eq: complexity}). }
    \label{fig: Complexity vs entropy finite alpha}
\end{figure}

\subsection{Results in the low $\alpha$ limit}
In this section, we focus on the planted model in the limit $\alpha\ll 1$. Indeed, it can be shown in this case that the non-trivial saddle point of the planted free energy verifies $\hat{q}\ll\hat{m}$, $q-m^2\ll 1-m$ and $1-m=o(1)$. We start first by rewriting the energetic contribution as
\begin{align}
    \alpha\int \mathcal{D}z\,  dw P_{\kold}[w]\,\phi^{\knew}_{\rm out}[w,z,q,m]&=\alpha\int \mathcal{D}B\, \frac{1}{2\,\mathcal{N}_{\kold}}\left\{\erf{\frac{\sqrt{q}\kold+mB}{\sqrt{q-m^2}}}+\erf{\frac{\sqrt{q}\kold-mB}{\sqrt{q-m^2}}}\right\}\\
    &\hspace{1.3cm}\times \log\left\{\frac{1}{2}\erf{\frac{\knew+B}{\sqrt{1-q}}}+\frac{1}{2}\erf{\frac{\knew-B}{\sqrt{1-q}}}\right\}\nonumber
\end{align}
with the change of variable
\begin{align}
    B=\frac{mw+\sqrt{q-m^2}z}{\sqrt{q}}\, , \quad B^\top=\frac{-mz+\sqrt{q-m^2}w}{\sqrt{q}}
\end{align}
 and integrating over $B^\top$. Recalling the fact that we focus on probing a saddle-point with $q-m^2\ll1-m$ and $1-m=o(1)$ (i.e. $1-q=o(1)$), we have for the energetic contribution
 \begin{align}
 \label{eq: saddle m low alpha}
    \alpha\int \mathcal{D}z\,  dw P_{\kold}[w]\,\phi^{\knew}_{\rm out}[w,z,q,m]&=\alpha\int \mathcal{D}B\, \frac{1}{2\,\mathcal{N}_{\kold}}\left\{\erf{\frac{\kold+B}{\sqrt{2(q-m^2)}}}+\erf{\frac{\kold-B}{\sqrt{2(q-m^2)}}}\right\}\\
    &\hspace{1.2cm}\times \log\left\{\frac{1}{2}\erf{\frac{\knew+B}{\sqrt{2(1-q)}}}+\frac{1}{2}\erf{\frac{\knew-B}{\sqrt{2(1-q)}}}\right\}+\mathcal{O}(1-q)\nonumber\, .
\end{align}
and the saddle point equations over $q$ and $m$ become
\begin{align}
    \hat{m}&= \frac{2\alpha m}{\mathcal{N}_{\kold}} \int \mathcal{D}B\left\{\frac{(\kold+B)e^{\frac{-(\kold+B)^2}{2(q-m^2)}}}{\sqrt{2\pi}(q-m^2)^{3/2}}+\frac{(\kold-B)e^{\frac{-(\kold-B)^2}{2(q-m^2)}}}{\sqrt{2\pi}(q-m^2)^{3/2}}\right\}\log\left\{\frac{1}{2}\erf{\frac{\knew+B}{\sqrt{2(1-q)}}}+\frac{1}{2}\erf{\frac{\knew-B}{\sqrt{2(1-q)}}}\right\}\, \\
    &+\mathcal{O}\left(\frac{1}{\sqrt{1-q}}\right)\nonumber,
\end{align}
\begin{align}
\label{eq: saddle q low alpha}
    \hat{q}=& -\frac{\alpha}{\mathcal{N}_{\kold}}  \int \mathcal{D}B\left\{\frac{(\kold+B)e^{\frac{-(\kold+B)^2}{2(q-m^2)}}}{\sqrt{2\pi}(q-m^2)^{3/2}}+\frac{(\kold-B)e^{\frac{-(\kold-B)^2}{2(q-m^2)}}}{\sqrt{2\pi}(q-m^2)^{3/2}}\right\}\log\left\{\frac{1}{2}\erf{\frac{\knew+B}{\sqrt{2(1-q)}}}+\frac{1}{2}\erf{\frac{\knew-B}{\sqrt{2(1-q)}}}\right\}\\
    &+\frac{\alpha}{\mathcal{N}_{\kold}}\int \mathcal{D}B\, \frac{\erf{\frac{\kold+B}{\sqrt{2(q-m^2)}}}+\erf{\frac{\kold-B}{\sqrt{2(q-m^2)}}}}{\erf{\frac{\knew+B}{\sqrt{2(1-q)}}}+\erf{\frac{\knew-B}{\sqrt{2(1-q)}}}}\left\{\frac{(\knew+B)e^{\frac{-(\knew+B)^2}{2(1-q)}}}{\sqrt{2\pi}(1-q)^{3/2}}+\frac{(\knew-B)e^{\frac{-(\knew-B)^2}{2(1-q)}}}{\sqrt{2\pi}(1-q)^{3/2}}\right\}+\mathcal{O}\left(\frac{1}{\sqrt{1-q}}\right)\nonumber\\
    =&-\frac{\hat{m}}{2m}+\frac{\alpha}{\mathcal{N}_{\kold}}\int \mathcal{D}B\, \frac{\erf{\frac{\kold+B}{\sqrt{2(q-m^2)}}}+\erf{\frac{\kold-B}{\sqrt{2(q-m^2)}}}}{\erf{\frac{\knew+B}{\sqrt{2(1-q)}}}+\erf{\frac{\knew-B}{\sqrt{2(1-q)}}}}\left\{\frac{(\knew+B)e^{\frac{-(\knew+B)^2}{2(1-q)}}}{\sqrt{2\pi}(1-q)^{3/2}}+\frac{(\knew-B)e^{\frac{-(\knew-B)^2}{2(1-q)}}}{\sqrt{2\pi}(1-q)^{3/2}}\right\}+\mathcal{O}\left(\frac{1}{\sqrt{1-q}}\right)\nonumber\, .
\end{align}
Now with $q-m^2=o(1)$ the energetic contribution becomes
 \begin{align}
    \frac{\alpha}{\mathcal{N}_{\kold}}\int \mathcal{D}z\,  dw P_{\kold}[w]\,\phi^{\knew}_{\rm out}[w,z,q,m]=&\frac{\alpha}{\mathcal{N}_{\kold}}\int \mathcal{D}B\, \Theta(\kold-\vert B\vert)\,\log\left\{\frac{1}{2}\erf{\frac{\knew+B}{\sqrt{2(1-m^2)}}}+\frac{1}{2}\erf{\frac{\knew-B}{\sqrt{2(1-m^2)}}}\right\}\\
    &+\mathcal{O}\left({q-m^2}\right)\nonumber
\end{align}
and we obtain using the saddle point with respect to $m$
\begin{align}
\label{eq: saddle appendix 1}
\hat{m}=\frac{2\alpha m}{\mathcal{N}_{\kold}}\int \mathcal{D}B\frac{\Theta(\kold-\vert B\vert)}{\frac{1}{2}\erf{\frac{\knew+B}{\sqrt{2(1-m^2)}}}+\frac{1}{2}\erf{\frac{\knew-B}{\sqrt{2(1-m^2)}}}}   \left\{\frac{(\knew+B)e^{\frac{-(\knew+B)^2}{2(1-m^2)}}}{\sqrt{2\pi}(1-m^2)^{3/2}}+\frac{(\knew-B)e^{\frac{-(\knew-B)^2}{2(1-m^2)}}}{\sqrt{2\pi}(1-m^2)^{3/2}}\right\}+\mathcal{O}\left(\frac{q-m^2}{(1-m^2)^{3/2}}\right)\, .
\end{align}
If we inject this solution in Eq.~(\ref{eq: saddle q low alpha}) we obtain $\hat{q}=o(1/(1-m)^{3/2})$, while Eq.~(\ref{eq: saddle appendix 1}) implies $\hat{m}=\mathcal{O}(1/(1-m)^{3/2})$. Finally, setting $\hat{q}\ll\hat{m}$ in Eq.~(\ref{eq: SE m}, \ref{eq: SE q}) we can derive 
\begin{align}
     q&=\int \mathcal{D}z\, {\rm tanh}^2(\hat{m}+\sqrt{\hat{q}}z)
      \hspace{-0.1cm}\underset{\hat{q}\ll\hat{m}}{=}{\rm tanh}^2(\hat{m})
    \, ,\\
    \label{eq: saddle appendix 2}
      m&=\int \mathcal{D}z\, {\rm tanh}(\hat{m}+\sqrt{\hat{q}}z)
      \hspace{-0.1cm}\underset{\hat{q}\ll\hat{m}}{=}{\rm tanh}(\hat{m})\, .
\end{align}
In a nutshell, by setting $q-m^2\ll1-m$ and $1-m=o(1)$ we obtain $\hat{m}=\mathcal{O}(1/(1-m)^{3/2})$ and $\hat{q}=o(1/(1-m)^{3/2})$. Then, we showed that $\hat{q}\ll \hat{m}$ implies $q-m^2\ll1-m$. At this stage we still need to check if closing Eqs. (\ref{eq: saddle appendix 1}) and (\ref{eq: saddle appendix 2}) will indeed provide $1-m=o(1)$

In fact, using $q-m^2\ll 1-m$, $1-m=o(1)$, $\hat{q}\ll\hat{m}$ along with Eq.~(\ref{eq: saddle appendix 2}) we have
\begin{align}
    \phi^{\kold,\knew}_{\rm planted}[m]=&-\frac{1-m}{2}\log{\left[\frac{1-m}{2}\right]}+\frac{\alpha }{\mathcal{N}_{\kold}}\int\! \mathcal{D}B\, \Theta(\kold-\vert B\vert)\,\log\left\{\frac{1}{2}\erf{\frac{\knew+B}{\sqrt{2(1-m^2)}}}+\frac{1}{2}\erf{\frac{\knew-B}{\sqrt{2(1-m^2)}}}\right\}\\
    &+o\left(\frac{1-m}{2}\log{\left[\frac{1-m}{2}\right]}\right)\,\nonumber.
\end{align}
Then, the planted free energy is a non-trivial function for only a restricted range of parameters $\knew$ and $m$. It happens when the entropic and energetic contributions compete with each other. This leads us to introduce a rescaling of the form $1-{m}^2= -\alpha  \tilde{r} /\log(\alpha)$, $\kold={\tkold} \sqrt{-\alpha /\log(\alpha)}$ and $\knew={\tknew} \sqrt{-\alpha /\log(\alpha)}$. For $\alpha\ll 1$, this rescaling enables to check directly that $1-m=o(1)$ as we have $\alpha/\log(\alpha)\ll1$. In other words, if there exist a solution for the saddle-point equations (\ref{eq: saddle appendix 1}) and (\ref{eq: saddle appendix 2}) it will be for $1-m=o(1)$ in the low $\alpha$ limit. If we now rewrite the above planted free energy with this  rescaling we obtain
\begin{align}
    \phi^{\tkold,\tknew}_{\rm planted}[\tilde{r}]&= \frac{ \alpha\tilde{r}}{4}+\frac{\alpha }{\tilde{\mathcal{N}}_{\kold}}\int\! \mathcal{D}B\, \Theta(\tkold-\vert B\vert)\,\log\left\{\frac{1}{2}\erf{\frac{\tknew+B}{\sqrt{2\tilde{r}}}}+\frac{1}{2}\erf{\frac{\tknew-B}{\sqrt{2\tilde{r}}}}\right\}+o(\alpha)\, ,\\
\label{eq: local max planting rescaled app}
   1&=\frac{4}{\tilde{\mathcal{N}}_{\kold}}\int \mathcal{D}B\frac{\Theta(\tkold-\vert B\vert)}{\erf{\frac{\tknew+B}{\sqrt{2\tilde{r}}}}+\erf{\frac{\tknew-B}{\sqrt{2\tilde{r}}}}}   \left\{\frac{(\tknew+B)e^{\frac{-(\tknew+B)^2}{2\tilde{r}}}}{\sqrt{2\pi}\tilde{r}^{3/2}}+\frac{(\tknew-B)e^{\frac{-(\tknew-B)^2}{2\tilde{r}}}}{\sqrt{2\pi}\tilde{r}^{3/2}}\right\}+o(1)\, 
\end{align}
with
\begin{align}
    \tilde{\mathcal{N}}_{\kold}=\int \mathcal{D}B\, \Theta(\tkold-\vert B\vert)\,.
\end{align}

A last simplification can be made if we set $\tkold=0$, for example this condition is verified when planting at $\kappa_{\sSAT}$ as $\kappa_{\sSAT}\sqrt{-\log(\alpha)/\alpha}\rightarrow 0$ when $\alpha$ is sent to zero. In this case the integration over $B$ can be dropped and we obtain for the local entropy and its saddle-point equation
\begin{align}
    \phi^{\tkold,\tknew}_{\rm planted}[\tilde{r}]=\frac{\tilde{r}}{4}+\alpha\log\left\{   \erf{\frac{\tknew}{\sqrt{2\tilde{r}}}} \right\}+o(\alpha)\, ,\\
     1=\frac{4\tknew e^{\frac{-\tknew^2}{2\tilde{r}}}}{\sqrt{2\pi}\tilde{r}^{3/2}\erf{\frac{\tknew}{\sqrt{2\tilde{r}}}}}+o(1) \, .
\end{align}

As a final note, and in order to draw a parallel with the 1-RSB computation, we outline here that the saddle-point equations (\ref{eq: saddle appendix 1}) and (\ref{eq: saddle appendix 2}) become after setting $m\rightarrow 1$  ($\hat m\gg1$) and $\kold=\kappa_{\sSAT}\rightarrow0$
\begin{align}
  \label{eq: saddle point m approx 1}
    m&= 1-2e^{-2\hat{m}}\, ,\\
    \label{eq: saddle point m approx 2}
    \hat{m}&=\frac{2\alpha\knew m}{\sqrt{2\pi(1-m^2)}(1-m^2)\erf{\frac{\knew}{\sqrt{2(1-m^2)}}}}e^{\frac{-\knew^2}{2(1-m^2)}}\,.
\end{align}

\end{document}